\documentclass[preprint,12pt,authoryear]{elsarticle}
\usepackage{amsmath,amsthm,amsfonts,amssymb,color}
\usepackage{amssymb}
\usepackage{amsthm}
\newtheorem{theorem}{Theorem}
\newtheorem{lemma}{Lemma}
\newtheorem{remark}{Remark}
\newtheorem{corollary}{Corollary}

\journal{Zeitschrift fur Angewandte Mathematik und Physik}

\begin{document}

\begin{frontmatter}



\title{Free boundary problem governed by a non-linear diffusion-convection equation with Neumann condition.}


\author{Adriana C. Briozzo }

\address{Depto. Matematica and CONICET, Universidad Austral\\
Paraguay 1950, S2000FZF Rosario, Argentina\\abriozzo@austral.edu.ar}

\begin{abstract}
We consider a one-dimensional free boundary problem governed by a nonlinear diffusion - convection equation with a Neumann condition at fixed face $x=0$, which is variable in time and a like Stefan convective condition on the free boundary. Through successive transformations, an integral representation of the problem is obtained that involves a system of coupled nonlinear integral equations. Existence of the solution is obtained for all times by using fixed point theorems.

\end{abstract}

\begin{keyword}
Diffusion-convection equation\sep  free boundary problem\sep nonlinear integral equation\sep Neumann condition 

MSC: \ 35R35  \ 45D05  \ 35K55 

\end{keyword}

\end{frontmatter}


\section{Introduction}

It is known in the literature that numerous phenomena are modeled through free boundary problems which are defined in domains with boundary or part of its boundary unknown and variable in time and that are governed by partial differential equations, among them the convection-diffusion equations \cite{AlSo,Ca,Ru,Ta}.  The filtration process of a fluid in a porous medium occurs under the effect of convective and diffusive forces \cite{PoMo,WWE}. The nonlinear evolution equation, called the Rosen-Fokas-Yorstos equation \cite{BuDeLi2012,FoYo,Ro}, is a nonlinear convection-diffusion equation suitable for describing the diffusion of fluids with convective effects in porous media \cite{RoBr2020}. This equation has multiple applications, for example, to ground water hydrology, oil reservoir engineering and other biological applications as the drug propagation in the arterial tissues.

In \cite{BrTa2019,BuDeLiFi2018} a one dimensional free boundary problem  for a nonlinear diffusion-convection equation with a constant Dirichlet condition on the fixed face $x=0$,  which describe drug diffusion in arterial tissues after the drug is released by an arterial stent was considered. The problem is solved through the reduction to a system of nonlinear integral equations. In \cite{BrTa2020} a time-varying Dirichlet type condition on the fixed face $x=0$ was considered for the free boundary problem and an explicit parametric representation of the solution was obtained.

In this paper we will analyze a one-dimensional free boundary problem governed by a diffusion-convection equation with a Neumann condition at the fixed face $x=0$ which is variable in the time. On the free boundary its assumed a  Stefan like condition of the convective type. 
Unlike the variable Dirichlet condition indicates how the concentration varies over time, on the fixed face, the Neumann condition with variable data in time allows us to know the flux of the fluid that crosses that fixed face.

This paper is organized as follows: in Section 2 we present the free boundary problem for a nonlinear diffusion-convection equation and through successive transformations we take this problem to one governed by the classical heat-diffusion equation. In Section 3, the integral formulation of the free boundary problem is presented, which consists of a system of coupled nonlinear integral equations. In Section 4, by using fixed point theorems, it is obtained that under certain conditions on the data there is a local solution in time, which can be expressed parametrically and  to be extended for all times. 

\section{Free boundary problem with time-varying Neumann condition}
We define the following free boundary problem which is governed by a diffusion-convective equation, for which we must determine the concentration $u=u(x,t)$ and the free boundary given by $s=s(t)$
\begin{equation}
u_{t}=u^{2}(Du_{xx}-u_{x})\;\;\;,\;\;0<x<s(t)\;\;,\;\;t>0\;\;, \label{calor1}
\end{equation}
\begin{equation}
Du_{x}(0,t)=g(t)\;,\;\;t>0\;\;,  \label{calor110}
\end{equation}
\begin{equation}
u(s(t),t)=\beta>0\;,\;\;t>0\;\;,  \label{tf}
\end{equation}
\begin{equation}
Du_{x}(s(t),t)-u(s(t),t)=-\dot{s}(t)\;,\;\;t>0\;\;,\label{stefan}
\end{equation}
\begin{equation}
u(x,0)=u_{0}(x)\geq \beta \;,\;\;0\leq x \leq b\;\;,  \label{t2}
\end{equation}
\begin{equation}
s(0)=b.\label{tempborde}
\end{equation}
With $D$ we name the diffusivity, $u_{0}=u_{0}(x)$ is the initial concentration and  $g=g(t)$ is the flux concentration in the fixed face $x=0$.
We assume that:
\begin{equation}
g\in C^{1}[0,\sigma],\quad u_{0}\in C^{1}[0,b],\quad Du'_{0}(0)=g(0),\quad u_{0}(b)=\beta\label{hip}
\end{equation}
In what follows, through successive transformations we will obtain a free boundary problem equivalent to the given one, which is governed by the classical heat equation. 
\begin{lemma}
Let $(u=u(x,t),s=s(t))$ be a solution to the free boundary problem (\ref{calor1})-(\ref{tempborde}) then $(v(z,t), z_{0}(t), z_{1}(t))$ given by:
\begin{equation}
v(z,t)=u(x,t),\label{firsttrans}
\end{equation}
\begin{equation}
z(x,t)= C_{1}+\int_{0}^{t}\left(u(0,\tau)-g(\tau)\right)d\tau + \int_{0}^{x}\frac{1}{u(\eta,t)}d\eta\label{zeta},
\end{equation}
\begin{equation}
z_{0}(t)=z(0,t), \quad\quad z_{1}(t)=z(s(t),t)
\end{equation}
with $C_{1}$ an arbitrary constant, is a solution to the  following free boundary problem governed by a Burgers equation 
\begin{equation}
v_{t}=Dv_{zz}-2vv_{z}\;\;\;,\;\;z_{0}(t)<z<z_{1}(t)\;\;,\;\;t>0\;\;, \label{calor11}
\end{equation} 
\begin{equation}
D\frac{v_{z}(z_{0}(t),t)}{v(z_0(t),t)}=g(t)\;,\;\;t>0\;\;,  \label{calor111}
\end{equation}
\begin{equation}
v(z_{1}(t),t)=\beta\;,\;\;t>0\;\;,  \label{tf1}
\end{equation}
\begin{equation}
Dv_{z}(z_{1}(t),t)-\beta^{2}=-\frac{\beta^{2}}{\beta +1}\dot{z}_{1}(t)\;,\;\;t>0\;\;,\label{stefan1}
\end{equation}
\begin{equation}
v(z,0)=v_{0}(z)\;,\;\;C_{1}\leq z\leq C_{2}\;\;,  \label{t21}
\end{equation}
\begin{equation}
z_{0}(0)=C_{1}\;,\;\;z_{1}(0)=C_{2}=i(b) \label{tempborde1}
\end{equation}

where 
\begin{equation} \label{esta}
v_{0}(z)=u_{0}(i^{-1}(z)),\quad\quad i(x)=C_{1}+\int^{x}_{0}\frac{1}{u_{0}(\eta)}d\eta
\end{equation}
and the constants $b$, $C_{1}$ and $C_{2}$ satisfy the following relation
\begin{equation}\label{be1}
b=\int^{C_{2}}_{C_{1}}v_{0}(z) dz
\end{equation}
Moreover the free boundaries $z_0=z_0(t)$ and $z_1=z_1(t)$ satisfy the integral equations given by
    \begin{equation}
z_{0}(t)=C_{1}+\int^{t}_{0} (v(z_{0}(\tau),\tau)-g(\tau))d\tau\label{zeta0}
\end{equation}

\begin{equation}
z_{1}(t)=C_{2}+(\beta+1)t-\frac{D(\beta +1)}{\beta^{2}}\int^{t}_{0} v_{z}(z_{1}(\tau),\tau)d\tau\label{zeta1}
\end{equation}
\end{lemma} 
\begin{proof}
 From $(\ref{firsttrans})$, $(\ref{zeta})$ and taking into account $(\ref{calor1})$ we have 
\[
z_{x}=\tfrac{1}{u(x,t)}=\tfrac{1}{v(z,t)},\quad z_{t}=u(x,t)-Du_{x}(x,t)=v(z,t)-D\tfrac{v_{z}(z,t)}{v(z,t)},
\]
and
\[ u_{x}(x,t)=\tfrac{v_{z}(z,t)}{v(z,t)},\quad u_{xx}(x,t)=\tfrac{v_{zz}(z,t)}{v^{2}(z,t)}-\tfrac{v^{2}_{z}(z,t)}{v^{3}(z,t)}, \]

\[ u_{t}(x,t)=v_{t}(z,t)+v_{z}\left(v(z,t)-D\tfrac{v_{z}(z,t)}{v(z,t)}\right).
\]
Then, from (\ref{calor1}) we get the Burgers equation (\ref{calor11}) for variable $v(z,t)$.

 The domain $D_u=\left\lbrace(x,t)/0<x<s(t), t>0\right\rbrace$ for $u=u(x,t)$ becomes into the domain $D_v=\left\lbrace(z,t)/z_{0}(t)<z<z_{1}(t),t>0)\right\rbrace$ for $v=v(z,t)$, where $z_{0}(t)$ and $z_{1}(t)$ are given by \[
z_{0}(t)=z(0,t)=C_{1}+\int_{0}^{t}\left(u(0,\tau)-g(\tau)\right)d\tau
\]and 
\[
z_{1}(t)=C_{1}+\int_{0}^{t}\left(u(0,\tau)-g(\tau)\right)d\tau + \int_{0}^{s(t)}\frac{1}{u(\eta,t)}d\eta.\]
Derivating the previous expression $z_{1}$ respect to $t$ and by using $(\ref{calor1})$-$(\ref{tempborde})$, we obtain the next relationship
\[
\dot{z}_{1}(t)=\tfrac{\beta+1}{\beta}\dot{s}_{1}(t)
\]
then, from $(\ref{stefan})$ we have $(\ref{stefan1}),$ moreover
\[
z_{0}(t)=C_{1}+\int^{t}_{0} (v(z_{0}(\tau),\tau)-g(\tau))d\tau 
\] and 
\[
z_{1}(t)=C_{2}+(\beta+1)t-\frac{D(\beta +1)}{\beta^{2}}\int^{t}_{0} v_{z}(z_{1}(\tau),\tau)d\tau\] satisfy the integral equations \eqref{zeta0} and \eqref{zeta1}.

Equations $(\ref{calor111})$ and $(\ref{tf1})$ follows immediately from $(\ref{calor110})$ and $(\ref{tf})$ respectively.

From \eqref{t2} and \eqref{zeta}, for $t=0$ we have that
\[z=C_{1}+ \int_{0}^{x}\frac{1}{u_{0}(\eta)}d\eta=i(x),
\] 
then $(\ref{t2})$ is equivalent to $v(z,0)=u_{0}\left(i^{-1}(z)\right)$ for $C_{1}\leq z \leq C_{2}$ where $C_{2}=i(b)=C_{1}+ \int_{0}^{b}\frac{1}{u_{0}(\eta)}d\eta$. Therefore $(\ref{t21})$ holds. Moreover
$$b=\int_{0}^{b}d\eta=\int_{C_1}^{C_2}(i^{-1})'(z)d\zeta
=\int_{C_1}^{C_2}v_0(z)dz.$$
\end{proof}
\begin{lemma}
Assuming that $(v(z,t),z_{0}(t),z_{1}(t))$ is a solution to the problem  $(\ref{calor11})-(\ref{tempborde1})$ then $(u(x,t), s(t))$ given by    
\begin{equation}
u(x,t)=v(z,t),\label{secondtrans}
\end{equation}
\begin{equation}
x(z,t)= \int_{z_{0}(t)}^{z} v(\eta,t)d\eta,\label{equis}
\end{equation}
\begin{equation}
s(t)=x(z_{1}(t),t)=\int_{z_{0}(t)}^{z_{1}(t)} v(\eta,t)d\eta\
\label{yy}\end{equation} is a solution to the problem $(\ref{calor1})-(\ref{tempborde})$.
\end{lemma}

\begin{proof}
We consider the definitions $(\ref{secondtrans})-(\ref{yy})$ and the conditions $(\ref{calor11})-(\ref{tempborde1})$ which are satisfied by $v=v(z,t)$, $z_{0}(t)$ and $z_{1}(t)$. We have
\[x_{z}=v(z,t), \quad x_{t}=Dv_{z}-v^{2}(z,t).
\]
Moreover, for $z=z_{0}(t)$ is $x=0$ and for $z=z_{1}(t)$ is $x=\int_{z_{0}(t)}^{z_{1}(t)} v(\eta,t)d\eta=s(t).$
Since
\[
v_{z}=u_{x}u, \quad v_{t}=Du^{2}_{x}u-u_{x}u^{2}+u_{t} ,\quad v_{zz}=u_{xx}u^{2}+u^{2}_{x} u
\]
then $(\ref{calor11})$ yields $(\ref{calor1}).$ 

The conditions $(\ref{calor110})$, $(\ref{tf})$ and $(\ref{t2})$ follows immediately from $(\ref{calor111})$, $(\ref{tf1})$ and $(\ref{t21})$ respectively.

To prove $(\ref{stefan})$, from $(\ref{yy})$ we calculate $\dot{s}(t)$ and use  $(\ref{calor11})$,  $(\ref{tf1})$, \eqref{stefan1} and \eqref{zeta0}. We have
\[
\dot{s}(t)=v(z_{1}(t),t)\dot{z}_{1}(t)-v(z_{0}(t),t)\dot{z}_{0}(t)+\int_{z_{0}(t)}^{z_{1}(t)} v_t(\eta,t)d\eta\]
\[=\beta-D\frac{v_{z}(z_{1}(t),t)}{\beta}= \beta-Du_{x}(s(t),t)
\]
and $(\ref{stefan})$ holds.
From \eqref{tempborde1}, \eqref{be1} and \eqref{yy} follows \eqref{tempborde}.
\end{proof}

\begin{remark}Eq. $(\ref{zeta})$ is equivalent to the reciprocal transformation 
\begin{equation}
dz=\frac{1}{u(x,t)}dx+(u(x,t)-Du_{x}(x,t))dt 
\end{equation}
and eq. $(\ref{equis})$ is equivalent to
\begin{equation}
dx=v(z,t)dz+(Dv_{z}-v^{2}(z,t))dt.
\end{equation}

\end{remark} 

\begin{lemma} Under the Galilean Transformation given by\begin{equation}
V(y,t)=v(z,t)-\beta, \quad\quad\quad y=z-2\beta t,\quad\quad t>0\label{2}
\end{equation} the problem (\ref{calor11})-(\ref{be1}) is equivalent to the following free boundary problem:
\begin{equation}
V_{t}=DV_{yy}-2VV_{y}\;\;\;,\;\;y_{0}(t)<y<y_{1}(t)\;\;,\;\;t>0\;\;, \label{cal}
\end{equation}
\begin{equation}
D\frac{V_{y}(y_{0}(t),t)}{V(y_0(t),t)+\beta}=g(t)\;,\;\;t>0\;\;,  \label{cal1}
\end{equation}
\begin{equation}
V(y_{1}(t),t)=0\;,\;\;t>0\;\;,  \label{t1}
\end{equation}
\begin{equation}
DV_{y}(y_{1}(t),t)=\frac{\beta^{2}(1-\beta-\dot{y}_{1}(t))}{\beta +1}\;,\;\;t>0\;\;,\label{ste1}
\end{equation}
\begin{equation}
V(y,0)=V_{0}(y)\;,\;\;C_{1}\leq y \leq C_{2}\;\;,  \label{tt21}
\end{equation}
\begin{equation}
y_{0}(0)=C_{1}\;,\;\;y_{1}(0)=C_{2} \label{temp}
\end{equation}
where the free boundaries are given by 
 \begin{equation}
y_{0}(t)=C_{1}-\beta t+\int^{t}_{0}(V(y_{0}(\tau),\tau)-g(\tau))d\tau\label{f}
\end{equation}

\begin{equation}
y_{1}(t)=C_{2}+(1-\beta)t-\frac{D(\beta +1)}{\beta^{2}}\int^{t}_{0} V_{y}(y_{1}(\tau),\tau)d\tau\label{fr}
\end{equation}with
\begin{equation}\label{ecbeta}
V_{0}(y)=v_{0}(y)-\beta
\end{equation} and 
\begin{equation}\label{be}
    b=\int_{C_1}^{C_2}V_0(y) dy-\beta(C_2-C_1).
\end{equation}
\end{lemma}
\begin{proof}
It is known that the Galilean transformation $(\ref{2})$ leaves invariant the Burgers equation $(\ref{calor11})$. The free boundaries $y_{0}$ and $y_{1}$ given by $(\ref{f})$-$(\ref{fr})$ are obtained from $(\ref{zeta0})$-$(\ref{zeta1})$. The conditions $(\ref{cal1})$-$(\ref{temp})$ are easily obtained from $(\ref{calor111})$-$(\ref{tempborde1})$ and \eqref{be} follow from \eqref{be1} .

The reciprocal is obtained in an analogous way
\end{proof}

Let us now transform problem $(\ref{cal})-(\ref{be})$ in the one which is governed by a classical heat-diffusion equation.
\begin{theorem}
Under the Hopf-Cole transformation given by \begin{equation}
w(y,t)=C(t)V(y,t)\eta(y,t), \quad y_{0}(t)\leq y \leq y_{1}(t)\;\;,\;\;t>0,\label{tres}
\end{equation}with
\begin{equation}
C(t)=exp\left(-\int^{t}_{0} V_{y}(y_{1}(\tau),\tau)d\tau\right),\label{ce}
\end{equation}and
\begin{equation}
\eta(y,t)=exp\left(\tfrac{1}{D}\int_{y}^{y_{1}(t)}V(\xi,t)d\xi\right)\label{eta}
\end{equation}
the problem $(\ref{cal})-(\ref{fr})$ is equivalent to the free boundary problem $(\ref{calu})-(\ref{free})$ given by:
\begin{equation}
w_{t}=Dw_{yy}\;\;\;,\;\;y_{0}(t)<y<y_{1}(t)\;\;,\;\;t>0\;\;, \label{calu}
\end{equation}
\begin{equation}
Dw_{y}(y_{0}(t),t)=g(t)w(y_0(t),t)+\beta g(t)w_0(t)-\tfrac{(w(y_{0}(t),t))^2}{w_0(t)}\;,\;\;t>0\;\;,  \label{cal1.}
\end{equation}
\begin{equation}
w(y_{1}(t),t)=0\;,\;\;t>0\;\;,  \label{t1.}
\end{equation}
\begin{equation}
\frac{Dw_{y}(y_{1}(t),t)}{\beta C(t)}=\frac{\beta(1-\beta)-\beta\dot{y}_{1}(t)}{\beta +1}\;,\;\;t>0\;\;,\label{ste1.}
\end{equation}
\begin{equation}
w(y,0)=V_0(y)\left[1+\tfrac{1}{D}\int_{y}^{C_{2}}w(\xi,0)d\xi\right]\;,\;\;C_{1}\leq y \leq C_{2}\;\;,  \label{tt21.}
\end{equation}
\begin{equation}
y_{0}(0)=C_{1}\;,\;\;y_{1}(0)=C_{2} \label{tempu.}
\end{equation}
where \begin{equation}\label{ce1}
C(t)=1-\int^{t}_{0} w_{y}(y_{1}(\tau),\tau)d\tau,
\end{equation}
\begin{equation}\label{w0}
w_{0}(t)=C(t)+\tfrac{1}{D}\int^{y_{1}(t)}_{y_{0}(t)}w(\xi,t)d\xi,
\end{equation}

and the free boundaries $y_{0}=y_{0}(t)$ and $y_{1}=y_{1}(t)$ are given by:
\begin{equation}\label{freee}
y_{0}(t)=C_{1}-\beta t-\int^{t}_{0} g(\tau)d\tau+\int^{t}_{0}\frac{w(y_{0}(\tau),\tau)}{w_{0}(\tau)}d\tau
\end{equation}
\begin{equation}
y_{1}(t)=C_{2}+(1-\beta)t+\frac{D(\beta +1)}{\beta^{2}}log\left(1-\int^{t}_{0} w_{y}(y_{1}(\tau),\tau)d\tau\right).
\label{free}
\end{equation}
Moreover
\begin{equation}
    b=\int_{C_1}^{C_{2}}\frac{w(y.0)}{1+\frac{1}{D}\int_{y}^{C_{2}}w(\xi,0)d\xi}dy-\beta(C_2-C_1).
\end{equation}
\end{theorem}

\begin{proof} Taking into account $(\ref{eta})$, we have
\[
log\left(\eta(y,t)\right)=\tfrac{1}{D}\int_{y}^{y_{1}(t)}V(\xi,t)d\xi
\] and
\[
\eta_{y}(y,t)=-\tfrac{1}{D}V(y,t)\eta(y,t)=-\tfrac{1}{D}\frac{w(y,t)}{C(t)}
.\]
It follows that
\[
\eta(y,t)=\frac{C(t)+\frac{1}{D}\int_{y}^{y_{1}(t)}w(\xi,t)d\xi}{C(t)}.
\]
From \eqref{ce} we have $$ ln(C(t))=-\int_{0}^{t}V_y(y_1(\tau),\tau) d\tau $$ therefore $$\frac{C'(t)}{C(t)}=V_y(y_1(t),t)=-\frac{w_y(y_1(\tau),\tau)}{C(t)}$$
and we obtain \eqref{ce1}.

Then, we have that the inverse relation to $(\ref{tres})$ is  
\begin{equation}
V(y,t)=\frac{w(y,t)}{C(t)+\frac{1}{D}\int_{y}^{y_{1}(t)}w(\xi,t)d\xi}.\label{inverse}
\end{equation}
Under transformation $(\ref{inverse})$ from equation $(\ref{cal})$ we obtain the classical heat-diffusion equation $(\ref{calu}).$ The initial and boundary conditions $(\ref{cal1.})-(\ref{tempu.})$ are obtained from $(\ref{cal1})-(\ref{temp})$. From $(\ref{f})$ and $(\ref{fr})$ we obtain the free boundaries 
 $(\ref{freee})$ and $(\ref{free})$ respectively.

The converse is proved analogously.
\end{proof}

\section{Integral formulation of the problem $(\ref{calu})-(\ref{free})$}
In this section we propose to obtain a system of coupled integral equations equivalent to the free boundary problem $(\ref{calu})-(\ref{free})$.

Let $K$ be the fundamental solution to the heat equation, defined by \begin{equation}
K\left( x,t,\xi ,\tau \right) =\left\{ 
\begin{array}{ll}
\frac{1}{2\sqrt{\pi D\left( t-\tau \right) }}\exp \left( -\frac{\left( x-\xi
\right) ^{2}}{4D\left( t-\tau \right) }\right) & t>\tau \\ 
0 & t\leq \tau
\end{array}
\right.  \label{defK}
\end{equation} and, $G$ and $N$ are the Green and Neumann functions respectively, for the upper half-plane $t>0$ given by
\begin{equation}
G\left( x,t,\xi ,\tau \right) =K\left( x,t,\xi ,\tau \right) -K\left(
-x,t,\xi ,\tau \right),  \label{defG}
\end{equation}
\begin{equation}
N\left( x,t,\xi ,\tau \right) =K\left( x,t,\xi ,\tau \right) +K\left(
-x,t,\xi ,\tau \right),  \label{defN}
\end{equation}
We have the following result: 
\begin{theorem}\label{theo}
Assuming $(\ref{hip})$ and $0<D<2$, the solution to the free boundary problem $(\ref{calu})-(\ref{free})$ has an integral representation given by
\begin{equation}
w(y,t)=\int\nolimits_{C_{1}}^{C_{2}}    N(y,t;\xi ,0)F(\xi)d\xi +D
\int_{0}^{t}\chi_{1}(\tau )N(y,t;y_{1}(\tau ),\tau )d\tau  \label{z}
\end{equation}
\[
+D\int_{0}^{t} \chi_{2}(\tau) N_{\xi}(y,t;y_{0}(\tau ),\tau ) d\tau +\beta \int_{0}^{t} \chi_{2}(\tau) N(y,t;y_{0}(\tau ),\tau ) d\tau 
\] 
\[-\beta \int_{0}^{t} g(\tau) w_0(\tau)N(y,t;y_{0}(\tau ),\tau ) d\tau\] where 
\begin{equation}\label{efe}
F(y)=V_{0}(y)exp\left(\frac{1}{D}\int_{y}^{C_{2}}V_0(\xi)d\xi\right)
\end{equation}
 and the free boundaries are given by
\begin{equation}
y_{0}(t)=C_{1}-\beta t-\int^{t}_{0} g(\tau)d\tau+\int^{t}_{0}\frac{\chi_{2}(\tau)}{w_{0}(\tau)}d\tau
\label{ycero}
\end{equation}
\begin{equation}
y_{1}(t)=C_{2}+(1-\beta)t+\frac{D(\beta +1)}{\beta^{2}}log\left(1-\int^{t}_{0} \chi_{1}(\tau) d\tau\right)
\label{ese}
\end{equation}
where 
\begin{equation}\label{w0}
w_{0}(t)=1-\int^{t}_{0} \chi_{1}(\tau)d\tau+\int^{y_{1}(t)}_{y_{0}(t)}w(\xi,t)d\xi,
\end{equation}
and $\chi_{1}$, $\chi_{2}$ are defined by 
\begin{equation}
\chi_{1}\left(t\right) = w_y
\left( y_{1}(t),t\right) \;\;,\;\;\chi_{2}\left( t\right) =w(y_{0}(t),t)\label{def}
\end{equation}
if and only if it satisfies the following system of two nonlinear integral
equations: 

\[\chi_{1}\left(t\right)
 =\frac{2}{2-D }\left\{G(y_{1}(t),t;C_1 ,0)F(C_1 )
+\int\nolimits_{C_{1}}^{C_{2}}G(y_{1}(t),t;\xi ,0)F'(\xi )d\xi\right.
\] \begin{equation}\label{ecintegralf} +D \int_{0}^{t}\chi_{1}(\tau )N_{y}(y_{1}(t),t;y_{1}(\tau ),\tau )d\tau +\int_{0}^{t} \chi_{2}(\tau)G_{\tau}(y_{1}(t),t;y_{0}(\tau ),\tau ) d\tau  
\end{equation} 
\[\left.+\beta \int_{0}^{t} \chi_{2}(\tau)N_{y}(y_{1}(t),t;y_{0}(\tau ),\tau ) d\tau-\beta\int_{0}^{t} g(\tau) w_0(\tau)N_y(y_1(t),t;y_{0}(\tau ),\tau ) d\tau \right\rbrace\]

\begin{equation}\label{ecintegralf1}
\chi_{2}\left(t\right) =\frac{2}{2-D}\left\lbrace\int\nolimits_{C_{1}}^{C_{2}}N(y_{0}(t),t;\xi ,0)F(\xi )d\xi ++ D
\int_{0}^{t}N (y_{0}(t),t;y_{1}(\tau ),\tau )\chi_{1}(\tau )d\tau  \right. 
\end{equation}
\[
 -D\int_{0}^{t}\chi_{2}(\tau)G_{y}(y_{0}(t),t;y_{0}(\tau ),\tau )d\tau+\beta\int_{0}^{t}\chi_{2}(\tau)N(y_{0}(t),t;y_{0}(\tau ),\tau )d\tau\]
 \[\left.-\beta\int_{0}^{t} g(\tau) w_0(\tau)N(y_0(t),t;y_{0}(\tau ),\tau ) d\tau\right\rbrace, 
\]
where 
functions $y_{0}\;$, $y_{1}$ are given by $\left( \ref{ycero}\right) $ and $%
\left( \ref{ese}\right) $ respectively and $w_0(t)$ is solution to the equation $\eqref{w0}$. 
\end{theorem} 
\begin{proof}
Let $(w (y,t), y_0 (t), y_1(t))$ be the solution to the problem $(\ref{calu})-(\ref{free})$ integrating the Green identity 
\begin{equation}
D\left( Nw _{\xi}-wN_{\xi }\right) _{\xi }-\left( Nw
\right) _{\tau }=0\;\;  \label{ngreen}
\end{equation}on the domain 
\[
D_{t,\epsilon}=\left\{ \left( \xi ,\tau \right) \text{ }/\text{ }%
y_{0}(\tau )<\xi <y_{1}\left( \tau \right) ,\text{ }\epsilon <\tau
<t-\epsilon \right\} \; 
\]
for $\epsilon >0$ and letting $\epsilon \rightarrow 0$, we obtain the integral representation for $w (y,t)\;$\cite{Fr1959,Ru} 
\begin{equation}
w(y,t)=\int\nolimits_{C_{1}}^{C_{2}}N(y,t;\xi ,0)w(\xi,0)
d\xi +D
\int_{0}^{t}w_{\xi}(y_{1}(\tau),\tau)N(y,t;y_{1}(\tau ),\tau )d\tau  \label{now}
\end{equation}
\[-\int_{0}^{t}w(y_{0}(\tau),\tau)\Dot{y}_0(\tau)N(y,t;y_{0}(\tau ),\tau )d\tau +D\int_{0}^{t} w(y_{0}(\tau),\tau) N_{\xi}(y,t;y_{0}(\tau ),\tau ) d\tau
\]
\[-D\int_{0}^{t}w_{\xi}(y_{0}(\tau),\tau)N(y,t;y_{0}(\tau ),\tau )d\tau .
\]
Defining $F(y)=w(y,0)$, from condition \eqref{tt21} we have $F$ must satisfy the integral equation\[
F(y)=V_{0}(y)\left(1+\frac{1}{D}\int_{y}^{C_{2}}F(\xi)d\xi\right). \label{ache}
\]If we derivate this equation we have $F$ satisfy the differential equation
\[F'(y)=\left(\frac{V_0'(y)}{V_0(y)}-\frac{1}{D}V_0(y)\right)F(y)\]whose solution is \begin{equation}\label{efe}
F(y)=V_{0}(y)exp\left(\frac{1}{D}\int_{y}^{C_{2}}V_0(\xi)d\xi\right).\end{equation}

By using the definitions of $\chi_{1}$ and $\chi_{2}$ given by $(\ref{def})$, and boundary conditions we have $(\ref{z}).$
If we differentiate $(\ref{z})$ in variable $y$ and we let $y\rightarrow y_{1}^{-}(t),$ by using
the jump relations \cite{Fr1959} we obtain the integral equation $%
\left( \ref{ecintegralf}\right)$. 

If we let $y\rightarrow y_{0}^{+}(t)$ in $(\ref{z})$ we obtain $%
\left( \ref{ecintegralf1}\right)$.

Conversely, the function $w(y,t)$ defined by (\ref{z}), where $%
\chi_{1} $ and $\chi_{2}$ are the solutions of $\left( \ref{ecintegralf}\right) $ and $%
\left( \ref{ecintegralf1}\right) ,$ satisfies the conditions $(\ref{calu})$,  $\eqref{ste1.}$ - $(\ref{tempu.})$. In order to
prove $(\ref{cal1.})$ and $(\ref{t1.})$ we define 
$$\mu _{1}\left( t\right) =w(y_{1}(t),t)$$ and 
$$\mu _{2}\left( t\right) =Dw_y(y_{0}(t),t)-g(t)w(y_{0}(t),t)+\frac{1}{D}\frac{w^{2}(y_{0}(t),t)}{w_0(t)}-\beta g(t)w_0(t).$$
If we integrate the Green identity\ (\ref{ngreen})$\;$over the domain $%
D_{t,\varepsilon }$and we let $\varepsilon
\rightarrow 0,$ we obtain that 
\[
w(y,t)=\int\nolimits_{C_{1}}^{C_{2}}N(y,t;\xi
,0)w(\xi,0)d\xi +D\int_{0}^{t} N(y,t;y_{1}(\tau ),\tau )w_y(y_1(\tau),\tau)d\tau 
\]
\[-D\int_{0}^{t} N_{y}(y,t;y_{1}(\tau ),\tau )w
(y_{1}(\tau ),\tau )d\tau +\int_{0}^{t} N(y,t;y_{1}(\tau ),\tau )w(y_{1}(\tau),\tau)y^{'}_{1}(\tau)d\tau 
\]
\[
-\int_{0}^{t}N(y,t;y_{0}(\tau ),\tau )\left[w(y_{0}(\tau),\tau)y^{'}_{0}(\tau) -Dw_y(y_{0}(t),t)\right]d\tau
\]
\begin{equation}
+D\int_{0}^{t} N_{\xi}(y,t;y_{0}(\tau ),\tau )w(y_{0}(\tau),\tau) d\tau. \label{zbis}
\end{equation}
\\
Then, if we compare this last expression (\ref{zbis}) with (\ref{z}) we
deduce that 
\[
\int_{0}^{t}\mu_{1}(\tau)\left[N(y,t;y_{1}(\tau ),\tau )\Dot{y}_1(\tau)-DN_y(y,t;y_{1}(\tau ),\tau )\right]d\tau
\]
\[
-\int_{0}^{t}\mu_{2}(\tau)N(y,t;y_{0}(\tau ),\tau )d\tau=0
\]
By taking $y\rightarrow y_{1}^{-}(t)$ we have 
\begin{equation}
\mu _{1}(t)=\frac{2}{D}\left[\int_{0}^{t}\mu_{1}(\tau)\left[N(y_1(t),t;y_{1}(\tau ),\tau )\Dot{y}_1(\tau)-DN_y(y_1(t),t;y_{1}(\tau ),\tau )\right]d\tau  \right.\label{fiuno}
\end{equation}
\[\left.
-\int_{0}^{t}\mu_{2}(\tau)N(y_1(\tau),t;y_{0}(\tau ),\tau )d\tau\right]
\] derivating in varibale $y$ and taking $y\rightarrow y_{0}^{+}(t)$ by jump relation we have 
\begin{equation}
\mu _{2}(t)=2\left[\int_{0}^{t}\mu _{1}(\tau \left[ N_{y}(y_{0}(t),t;y_{1}(\tau ),\tau )\Dot{y}_1(\tau)-DN_{yy}(y_{0}(t),t;y_{1}(\tau ),\tau )\right] d\tau\right.\label{fidos}
\end{equation}
\[
\left.-\int_{0}^{t} N_{y}(y_{0}(t),t;y_{0}(\tau ),\tau) \mu _{2}(\tau) d\tau\right].
\]
Therefore we obtain that $\mu _{1}$
and $\mu _{2}\;$ must satisfy a system of Volterra integral
equations. From \cite{Mi}, it's easy to see that there exist a unique solution $%
\mu _{1}\equiv \mu _{2}\equiv 0$ to the system of Volterra integral
equations (\ref{fiuno})-(\ref{fidos}). Then $(\ref{cal1.})$ and $(\ref{t1.})$
are verified and the Theorem 3.1 holds.\medskip
\end{proof}
In the next section we will prove that there exist solution to the coupled nonlinear integral equations \eqref{ecintegralf}, \eqref{ecintegralf1} and \eqref{w0}.

\begin{section}{Existence of the solution}

Through the results given by Theorem \ref{theo} we can analyze the existence and uniqueness of the solution to the problem $(\ref{calu})-(\ref{free})$. Firstly we will consider the system of nonlinear integral equations $\left( \ref{ecintegralf}\right) $- $
\left( \ref{ecintegralf1}\right)$ where $yo$ and $y_1$ are given by \eqref{ycero} and \eqref{ese} respectively. 

To solve the system of integral equations $\left( \ref{ecintegralf}\right) $- $
\left( \ref{ecintegralf1}\right)$ we leave the function $w_0$ fixed in the set \begin{equation}\label{lambda} \Lambda=\Lambda(H,R)=\lbrace{w_0\in  C[0,\sigma]/0<H\leq w_0(t)\leq R, \forall t\in[0,\sigma]\rbrace},\end{equation} where $H$ and $R$ are positive constants which will be chosen  appropriately later.

We consider the Banach space \[
\textbf{C}[0,\sigma]=\left\{ \stackrel{\longrightarrow }{\chi^{*}}=\binom{\chi_{1}}{\chi_{2}}%
/\;\chi_{i}:\left[ 0,\sigma \right] \rightarrow {\Bbb R} ,\quad i=1,2,\;\text{continuous}
\right\} 
\] 
with the norm
\[
\left\| \stackrel{\longrightarrow }{\chi^{*}}\right\| _{\sigma }:=\max \limits_{
t\in \left[ 0,\sigma \right] }\left| \chi_{1}(t)\right| +\max\limits_{
t\in \left[ 0,\sigma \right] }\left| \chi_{2}(t)\right| 
\] and the set:
\[
C_{M,\sigma}=\left\{ \stackrel{\longrightarrow }{\chi^{*}}\in \textbf{C}[0,\sigma] /\left\| \stackrel{\longrightarrow }{\chi^{*}}\right\|_{\sigma }\leq
M\right\} 
\]
with $\sigma \,\,$ and $M$ positive numbers to be determinate. 
\\
We define the map $\Psi:C_{M,\sigma }\longrightarrow C_{M,\sigma },$ such that 
\[
\Psi\left( \stackrel{%
\longrightarrow }{\chi^{*}}\right)(t) =\binom{\Psi_{1}(\chi_{1}(t),\chi_{2}(t))}{\Psi_{2}(\chi_{1}(t),\chi_{2}(t))%
} 
\]
where the operators $\Psi_{i}$ are defined as follows
\[\Psi_{1}(\chi_{1}(t),\chi_{2}(t))=\frac{2}{2-D }\left\{G(y_{1}(t),t;C_1 ,0)F(C_1 )
+\int\nolimits_{C_{1}}^{C_{2}}G(y_{1}(t),t;\xi ,0)F'(\xi )d\xi\right.\]
\begin{equation}+ D \int_{0}^{t}\chi_{1}(\tau )N_{y}(y_{1}(t),t;y_{1}(\tau ),\tau )d\tau -\beta\int_{0}^{t} g(\tau) w_0(\tau)N_y(y_1(t),t;y_{0}(\tau ),\tau ) d\tau   
\end{equation} 
\[\left.+\int_{0}^{t} \chi_{2}(\tau)G_{\tau}(y_{1}(t),t;y_{0}(\tau ),\tau ) d\tau +\beta \int_{0}^{t} \chi_{2}(\tau)N_{y}(y_{1}(t),t;y_{0}(\tau ),\tau ) d\tau\right\rbrace,
\]

\begin{equation}
\Psi_{2}(\chi_{1}(t),\chi_{2}(t))= \frac{2}{2-D}\left\lbrace\int\nolimits_{C_{1}}^{C_{2}}N(y_{0}(t),t;\xi ,0)F(\xi )d\xi   \right. 
\end{equation}
\[
+ D
\int_{0}^{t}N (y_{0}(t),t;y_{1}(\tau ),\tau )\chi_{1}(\tau )d\tau-D\int_{0}^{t}\chi_{2}(\tau)G_{y}(y_{0}(t),t;y_{0}(\tau ),\tau )d\tau
\]
\[
\left. -\beta\int_{0}^{t} g(\tau) w_0(\tau)N(y_0(t),t;y_{0}(\tau ),\tau ) d\tau+\beta\int_{0}^{t}\chi_{2}(\tau)N(y_{0}(t),t;y_{0}(\tau ),\tau )d\tau\right\rbrace, 
\]
We will prove that the map $\Psi$ has a unique fixed point and then we will have proven that the system has only one solution. 

Below we state the following preliminaries
\begin{lemma}\label{cotasy}
Let $\sigma M<1$ and $\chi_{i}\in C^{0}\left[ 0,\sigma \right] ,\max \limits_ {t\in \left[ 0,\sigma
\right]}\left| \chi_{i}(t)\right| \leq M,(i=1,2)$. If $$2(1+\beta)\left(1+\frac{DM}{\beta^{2}}\right)\sigma \leq C_{2},\quad\quad 2\left(\beta+D||g|| +\frac{M}{H}\right)\sigma\leq C_{1}\;$$ then $y_{0}$ and $y_{1}$ defined by (\ref{ycero})$\;$%
and (\ref{ese}) satisfies 
\begin{equation}
\left| y_{0}(t)-y_{0}(\tau) \right| \leq \left(\beta+D||g||+\frac{M}{H}\right)\left| t-\tau
\right| \text{ \ ,\ }\forall \tau ,t\in \left[ 0,\sigma \right], \label{io}
\end{equation}
\begin{equation}
\tfrac{C_{1}}{2}\leq y_{0}(t) \leq 3\tfrac{C_{1}}{2},\forall
t\in \left[ 0,\sigma \right], \label{io1}
\end{equation}
\begin{equation}
 \left| y_{1}(t)-y_{1}(\tau) \right| \leq (1+\beta)\left(1+\tfrac{M}{\beta^{2}}\right)\left|t-\tau
\right| \text{ \ ,\ }\forall \tau ,t\in \left[ 0,\sigma \right],\label{io2}
\end{equation}
\begin{equation}
\tfrac{C_{2}}{2}\leq y_{1}(t)\leq 3\tfrac{C_{2}}{2},\text{ }%
\forall t\in \left[ 0,\sigma \right].\label{io3}
\end{equation}
\end{lemma}
\begin{proof} It follows inmediatly from definitions (\ref{ycero})-(\ref{ese}) and assumptions on data.
\end{proof}
\begin{lemma}\label{cotasyi}
Let $y_{01}$ and $y_{02}$ be the functions corresponding to $\chi_{21}$
and $\chi_{22}$ in $C^{0}[0,\sigma ]$ respectively, and $y_{11}$ and $%
y_{12}$ be the functions corresponding to $\chi_{11}$ and $\chi_{12}$ in $%
C^{0}[0,\sigma ]$ respectively with $
\max\limits_{t\in \left[ 0,\sigma \right]} \left| \chi_{ij}(t)\right| \leq
M,\,\quad i,j=1,2$. Under hypothesis of  Lemma \ref{cotasy} we have 
\begin{equation}
\left\{ 
\begin{array}{c}
\left| y_{01}(t)-y_{02}(t) \right| \leq
\frac{1}{H}\sigma\left\| \chi_{11}-\chi_{12}\right\| _{\sigma } ,\\ 
\\ 
\left| y_{01}(t)-y_{02}(\tau) \right| \leq  \left(\beta+D\|g\|+\tfrac{M}{H}\right)
\left| t-\tau \right| ,\text{ }i=1,2, \\ 
\\ 
\frac{C_{1}}{2}\leq y_{0i}(t) \leq \frac{3C_{1}}{2},\text{ }%
\forall t\in \left[ 0,\sigma \right] ,\text{ }i=1,2,
\end{array}
\right.
\end{equation}\label{desy}
and 
\begin{equation}
\left\{ 
\begin{array}{c}
\left| y_{11}(t)-y_{12}(t) \right| \leq
D\frac{\beta +1}{\beta^{2}}\sigma \left\| \chi_{21}-\chi_{22}\right\| _{\sigma } ,\\ 
\\ 
\left| y_{1i}(t)-y_{1i}(\tau) \right| \leq
 (1+\beta)\left(1+\tfrac{DM}{\beta^{2}}\right)\left| t-\tau \right| ,\text{ }i=1,2, \\ 
\\ 
\frac{C_{2}}{2}\leq y_{1i}(t) \leq \frac{3C_{2}}{2},%
\text{ }\forall t\in \left[ 0,\sigma \right] ,\text{ }i=1,2.
\end{array}
\right.
\end{equation}\label{desyi}
\end{lemma}

\begin{proof}
It follows immediately from definitions (\ref{ycero})-(\ref{ese}) and assumptions on data.
\end{proof}
To prove the following Lemmas we use the classical inequality 
\begin{equation}
\dfrac{\exp \left( \frac{-x^{2}}{\alpha \left( t-\tau \right) }\right) }{%
\left( t-\tau \right) ^{\frac{n}{2}}}\leq \left( \frac{n\alpha }{2ex^{2}}%
\right) ^{^{\frac{n}{2}}}\;,\;\alpha ,x>0\;,\;t>\tau \;,\;n\in {\Bbb N.}
\label{exp}
\end{equation}

\begin{lemma}\label{cotasint}
$\;$Let  $\sigma \leq\tfrac{1}{M} $, $\sigma \leq 1.$ Under the hypothesis of Lemma $\ref{cotasy}$ and $3C_{1}< C_2$ we have that following properties are satisfied
\begin{equation} \label{cota0}
    |G(y_{1}(t),t;C_1 ,0)F(C_1 )\leq E_0(C_1,C_2, u_0,D)\sqrt{\sigma}
\end{equation}
\begin{equation}
\int\nolimits_{C_{1}}^{C_{2}}\left| F^{\prime }(\xi )\right| \left|
G(y_{1}(t),t;\xi ,0)\right| d\xi \leq \left\| F^{\prime
}\right\| \leq E_{1}(u_{0},C_1, C_2,\beta, D), \label{i}
\end{equation}
\begin{equation}
D\int_{0}^{t}\left| N_{y}(y_{1}(t),t;y_{1}(\tau ),\tau )\chi_{1}(\tau
)\right| d\tau \leq E_{2}(M,D,\beta,C_{1},C_2)\;\sqrt{\sigma},  \label{olvidada}
\end{equation}
\begin{equation}
\int_{0}^{t}\left| G_{\tau}(y_{1}(t),t;y_{0}(\tau ),\tau )\chi_{2}(\tau
)\right| d\tau \leq E_{3}(M,C_{2},C_{1})\;\sqrt{\sigma},\label{olv1}
\end{equation}
\medskip 
\begin{equation}
\beta \int_{0}^{t}\left| N_{y}(y_{1}(t),t;y_{0}(\tau ),\tau )\chi_{2}(\tau
)\right| d\tau \leq E_{4}(M,\beta,C_{1},C_2)\;\sqrt{\sigma},\label{delhoy}
\end{equation}
\begin{equation}
\beta\int_{0}^{t} \left|g(\tau) w_0(\tau)N_y(y_1(t),t;y_{0}(\tau ),\tau ) \right|d\tau \leq E_{5}(R,\beta,C_{1},C_2)\;\sqrt{\sigma},\label{delhoy1}
 \end{equation}
\begin{equation}
\int\nolimits_{C_{1}}^{C_{2}}\left| F(\xi )\right| \left|
N(y_{0}(t),t;\xi ,0)\right| d\xi \leq \left\| F
\right\| \leq E_{6}(u_{0},C_1, C_2,\beta, D),  \label{iii}
\end{equation}
\medskip 
\begin{equation}
D\int_{0}^{t}\left| N(y_{0}(t),t;y_{1}(\tau ),\tau )\chi_{1}(\tau
)\right| d\tau \leq E_{7}(D,M)\sqrt{\sigma}, \label{olvidada1}
\end{equation}
\begin{equation}
D\int_{0}^{t}\left| G_{y}(y_{0}(t),t;y_{0}(\tau ),\tau )\chi_{2}(\tau
)\right| d\tau D\leq E_{8}(M,D,H,\beta,g, C_{1})\;\sqrt{\sigma},\label{delihoy}
\end{equation}
\begin{equation}
\beta\int_{0}^{t}\left| \chi_2(\tau )\right| \left|
N(y_{0}(t),t;y_{0}(\tau ),\tau )\right| d\tau \leq  E_{9}(\beta,M)\sqrt{\sigma}  ,\label{iibis}
\end{equation}
\begin{equation}
\beta\int_{0}^{t}\left|g(\tau )\right| w_0(\tau)\left|
N(y_{0}(t),t;y_{0}(\tau ),\tau )\right| d\tau \leq  E_{10}(\beta,R,g)\sqrt{\sigma}  ,\label{iiibis}
\end{equation}
\medskip where 
\[
E_0(C_1,C_2, u_0,D)=\tfrac{(\|u_0\|+\beta)^{2}(C_2-C_1)}{D}\left[\left(\tfrac{8}{e(C_2-2C_1)^{2})}\right)^{1/2}+\left(\tfrac{8}{e(C_2+2C_1)^{2})}\right)^{1/2}\right]
\]
\[E_{1}(u_{0},Ca-1,C_2,\beta, D)=\exp\left(\tfrac{\left\|(u_{0}\right\|+\beta)(C_2-C_1)}{D}\right) \left[\|u_{0}^{'}\|+\tfrac{(\left\|u_{0}\right\|+\beta)^{2}}{D} \right],
\]
\[
E_{2}(M,D,\beta,C_2)=\frac{DM}{4\sqrt{\pi}} \left[2(1+\beta)\left(1+\tfrac{M}{\beta^{2}}\right)+3C_2 \left(\frac{2}{3eC_2^{2}}\right)^{3/2}\right],
\]
\[
E_{3}(M;C_{1},C_{2})=M(A_{31}+A_{32)}),
\]
\[E_{31}=\tfrac{9(C_2+C_1)^{2}}{32\sqrt{\pi}}\left(\tfrac{40}{e(C_2-3C_1)^{2}}
\right)^{5/2}+\tfrac{1}{4\sqrt{\pi}}\left(\frac{24}{e(C_2-3C_1)^{2}}
\right)^{3/2}
\] \[
E_{32}=\tfrac{9(C_2+C_1)^{2}}{32\sqrt{\pi}}\left(\tfrac{40}{e(C_2+C_1)^{2}}
\right)^{5/2}+\tfrac{1}{4\sqrt{\pi}}\left(\tfrac{24}{e(C_2+C_1)^{2}}
\right)^{3/2}
\]

\[
E_{4}(M,\beta,C_1,C_{2})=\tfrac{3M\beta(C_1+C_2)}{8\sqrt{\pi}} \left[ \left(\tfrac{24}{e(C_2-3C_1)^{2}}\right)^{3/2}+\left(\tfrac{24}{e(C_2+C_1)^{2}}\right)^{3/2}\right],
\]
\[
E_{5}(R,\beta,g,C_1,C_{2})=\tfrac{3R\|g\|\beta(C_1+C_2)}{8\sqrt{\pi}} \left[ \left(\tfrac{24}{e(C_2-3C_1)^{2}}\right)^{3/2}+\left(\tfrac{24}{e(C_2+C_1)^{2}}\right)^{3/2}\right],
\]
\[
E_{6}(u_0,D, \beta, C_1, C_2)=\exp\left(\tfrac{(\left\|u_{0}\right\|+\beta )(C_2-C_1)}{D}\right) \left[\|u_{0}\|+\beta) \right],
\]
\[
E_{7}(D,M)=\frac{2DM}{\sqrt{\pi}}, 
\]\[
E_{8}(M,D,H, \beta,g,C_{1})=\tfrac{DM}{4\sqrt{\pi}} \left[2\left(\beta+D||g||+\tfrac{M}{H}\right)+3C_1 \left(\tfrac{2}{3eC_1^{2}}\right)^{3/2}\right],
\]

\[
E_{9}(\beta,M)=\frac{2\beta M}{\sqrt{\pi}}, 
\]
\[
E_{10}(\beta,R,g)=\frac{2\beta R \|g\|}{\sqrt{\pi}}, 
\]

\medskip 
 \end{lemma}

\begin{proof} 
By using \eqref{exp} follows \eqref{cota0}.

To prove (\ref{i}) we consider 
\[
\int\nolimits_{C_{1}}^{C_{2}}\left| F^{\prime }(\xi )\right| \left|
G(y_{1}(t),t;\xi ,0)\right| d\xi \leq \left\| F^{\prime
}\right\| \int_{0}^{\infty }\left| G(y_{1}(t),t;\xi ,0)\right|
d\xi \leq \left\| F^{\prime}\right\| .
\]
From $(\ref{efe})$ we have
\[
F'(y)=exp\left(\tfrac{1}{D}\int_{y}^{C_{2}}V_{0}(\xi)d\xi\right) \left[V^{'}_{0}(y)-\frac{1}{D}V_{0}^{2}(y)\right]
\] then
\[
\left\|F'\right\|\leq exp\left(\tfrac{\left\|V_{0}\right\|(C_{2}-C_{1})}{D}\right) \left[\|V_{0}^{'}\|+\tfrac{1}{D}\left\|V_{0}\right\|^{2}\right]\]
\[\leq \exp\left(\tfrac{(\left\|u_{0}\right\|+\beta) (C_{2}-C_{1})}{D}\right) \left[\|u_{0}^{'}\|+\tfrac{(\left\|u_{0}\right\|+\beta)^{2}}{D} \right]=E_{1}(u_{0},C_1,C_2,\beta, D).
\]
To prove \eqref{olvidada} we consider definition of $N$ and use the inequalities \eqref{io2}, \eqref{io3} and \eqref{exp}, then we have
\[\left| N_{y}(y_{1}(t),t;y_{1}(\tau ),\tau )\right|\leq\frac{(t-\tau)^{-3/2}}{4\sqrt{\pi}} \left[(y_1(t)-y_1(\tau))\exp\left(-\frac{(y_{1}(t)-y_{1}(\tau))^{2}}{4(t-\tau)}\right)\right.\]
\[\left.-(y_1(t)+y_1(\tau))\exp\left(-\frac{(y_{1}(t)+y_{1}(\tau))^{2}}{4(t-\tau)}\right)\right]\] 
\[\leq\frac{(t-\tau)^{-3/2}}{4\sqrt{\pi}} \left[(1+\beta)\left(1+\tfrac{M}{\beta^{2}}\right)\left|t-\tau
\right|+3C_2\exp\left(-\frac{9C_2^{2}}{4(t-\tau)}\right)\right]\]
\[\leq\frac{1}{4\sqrt{\pi}} \left[(1+\beta)\left(1+\tfrac{M}{\beta^{2}}\right)\left|t-\tau
\right|^{-1/2}+3C_2 \left(\frac{2}{3eC_2^{2}}\right)^{3/2}\right]\]
and it results \[D\int_{0}^{t}\left| N_{y}(y_{1}(t),t;y_{1}(\tau ),\tau )\chi_{1}(\tau
)\right| d\tau\leq\frac{DM\sqrt{\sigma}}{4\sqrt{\pi}} \left[2(1+\beta)\left(1+\tfrac{M}{\beta^{2}}\right)+3C_2 \left(\frac{2}{3eC_2^{2}}\right)^{3/2}\right]\]
Analogously by using definition of $G$, inequalities \eqref{io}, \eqref{io1} and \eqref{exp} we obtain
\[
D\int_{0}^{t}\left| G_{y}(y_{0}(t),t;y_{0}(\tau ),\tau )\chi_{2}(\tau
)\right| d\tau D\leq\frac{DM\sqrt{\sigma}}{4\sqrt{\pi}} \left[2\left(\beta+||g||+\frac{M}{H}\right)+3C_1 \left(\frac{2}{3eC_1^{2}}\right)^{3/2}\right]
\] this is
\eqref{delihoy}.

Inequality (\ref{iii}) is proved in the same way as (\ref{i}).

To prove \eqref{olv1} we taking into account
\[
G_\tau(y_{1}(t),t;y_{0}(\tau ),\tau )=\frac{\exp\left(-\frac{(y_{1}(t)-y_{0}(\tau))^{2}}{4(t-\tau)}\right)}{2\sqrt{\pi(t-\tau)}}\left[-\frac{(y_{1}(t)-y_{0}(\tau))^{2}}{4(t-\tau)^{2}}+\frac{1}{2(t-\tau)}
\right]
\]
\[
+\frac{\exp\left(-\frac{(y_{1}(t)+y_{0}(\tau))^{2}}{4(t-\tau)}\right)}{2\sqrt{\pi(t-\tau)}}\left[-\frac{(y_{1}(t)+y_{0}(\tau))^{2}}{4(t-\tau)^{2}}+\frac{1}{2(t-\tau)}
\right] 
\] and
\[y_{1}(t)-y_{0}(\tau)\leq \frac{3(C_2+C_1)}{2},\qquad y_{1}(t)+y_{0}(\tau)\leq \frac{3(C_2+C_1)}{2},\]
\[y_{1}(t)-y_{0}(\tau)\geq \frac{C_2-3C_1}{2},\qquad y_{1}(t)+y_{0}(\tau)\geq \frac{C_2+C_1}{2},\] then, by using \eqref{exp} we have 
\[
|G_\tau(y_{1}(t),t;y_{0}(\tau ),\tau )|\leq\frac{9(C_2+C_1)^{2}}{32\sqrt{\pi}}\left(\frac{40}{e(C_2-3C_1)^{2}}
\right)^{5/2}+\frac{1}{4\sqrt{\pi}}\left(\frac{24}{e(C_2-3C_1)^{2}}
\right)^{3/2}
\]
\[
+\frac{9(C_2+C_1)^{2}}{32\sqrt{\pi}}\left(\frac{40}{e(C_2+C_1)^{2}}
\right)^{5/2}+\frac{1}{4\sqrt{\pi}}\left(\frac{24}{e(C_2+C_1)^{2}}
\right)^{3/2}
\] therefore \eqref{olv1} holds.

To prove \eqref{delhoy} we proceed as in \eqref{olvidada} and \eqref{olv1} \[\left| N_{y}(y_{1}(t),t;y_{0}(\tau ),\tau )\right|\leq\frac{(t-\tau)^{-3/2}}{4\sqrt{\pi}} \left[(y_1(t)-y_0(\tau))\exp\left(-\frac{(y_{1}(t)-y_{0}(\tau))^{2}}{4(t-\tau)}\right)\right.\]
\[\left.-(y_1(t)+y_0(\tau))\exp\left(-\frac{(y_{1}(t)+y_{0}(\tau))^{2}}{4(t-\tau)}\right)\right]\] 
\[\leq\frac{(t-\tau)^{-3/2}}{4\sqrt{\pi}} \left[\frac{3(C_1+C_2)}{2}\exp\left(-\frac{(C_2-3C_1)^{2}}{16(t-\tau)}\right)+\frac{3(C_1+C_2)}{2}\exp\left(-\frac{(C_1+C_2)^{2}}{16(t-\tau)}\right)\right]\]
\[\leq\frac{3(C_1+C_2)}{8\sqrt{\pi}} \left[ \left(\frac{24}{e(C_2-3C_1)^{2}}\right)^{3/2}+\left(\tfrac{24}{e(C_2+C_1)^{2}}\right)^{3/2}\right]\]
and it results \[\beta\int_{0}^{t}\left| N_{y}(y_{1}(t),t;y_{0}(\tau ),\tau )\chi_{1}(\tau
)\right| d\tau\leq\tfrac{3\beta M\sqrt{\sigma}(C_1+C_2)}{8\sqrt{\pi}} \left[ \left(\tfrac{24}{e(C_2-3C_1)^{2}}\right)^{3/2}+\left(\tfrac{24}{e(C_2+C_1)^{2}}\right)^{3/2}\right].
\] Analogously we obtain the inequality \eqref{delhoy1}.

To prove \eqref{olvidada1} we take into account that 
\[
\left| N(y_{1}(t),t;y_{0}(\tau ),\tau )\right| \leq \frac{1}{%
\sqrt{\pi \left( t-\tau \right) }} 
\]
so, we obtain 
\[
D\int_{0}^{t}\left| N(y_{0}(t),t;y_{1}(\tau ),\tau )\chi_{1}(\tau
)\right| d\tau \leq \frac{2DM}{\sqrt{\pi}}\sqrt{\sigma}, 
\]
Analogously for \eqref{iibis} we have \[\beta\int_{0}^{t}\left| \chi_2(\tau )\right| \left|
N(y_{0}(t),t;y_{0}(\tau ),\tau )\right| d\tau \leq 
 \frac{2\beta M}{\sqrt{\pi}}\sqrt{\sigma}.   
\]
Inequality \eqref{iiibis} is obtained in the same way. 
\end{proof}

\begin{lemma}\label{cotasintyi}
If we take $\sigma \leq 1$, the hypothesis of Lemma \ref{cotasyi} and the following inequalities 
\begin{equation}
    \tfrac{M}{HD}\left(\beta + D\|g\|+\tfrac{M}{H}\right)\sigma \leq 1
\end{equation}
\begin{equation}
    \tfrac{2M(\beta+1)}{\beta^{2}}\left(1 + \tfrac{DM}{\beta^{2}}\right)\sigma \leq 1
\end{equation}
holds then we have 
\begin{equation}
 \left|
G(y_{11}(t),t;C_1 ,0)-G(y_{12}(t),t;C_1 ,0)\right| |F(C_1)|
\label{cotacero1}
\end{equation}
\[
\leq F_{0}(u_{0},D,\beta, C_1, C_2)\left\|
\chi_{11}-\chi_{12}\right\|_{\sigma} \sqrt{\sigma } \leq F_{0}(u_{0},D,\beta, C_1, C_2) \left\| \vec{\chi_{1}^{*}}-\vec{\chi_{2}^{*}}\right\|\sqrt{\sigma},
\]
\begin{equation}
\int_{C_{1}}^{C_{2}}\left| F^{^{\prime }}(\xi )\right| \left|
G(y_{11}(t),t;\xi ,0)-G(y_{12}(t),t;\xi ,0)\right| d\xi
\label{1iiprimero}
\end{equation}
\[
\leq \frac{2\left\|F^{\prime }\right\|_{[C_{1},C_{2}]}}{\sqrt{\pi}}\left\|
\chi_{11}-\chi_{12}\right\|_{\sigma} \sqrt{\sigma } \leq F_{1}(u_{0},D,\beta,C_1, C_2) \left\| \vec{\chi_{1}^{*}}-\vec{\chi_{2}^{*}}\right\|\sqrt{\sigma},
\]
\begin{equation}
D\int_{0}^{t}\left| \chi_{11}(\tau )N_{y}(y_{11}(t),t;y_{11}(\tau
),\tau )-\chi_{12}(\tau )N_{y}(y_{12}(t),t;y_{12}(\tau ),\tau )\right|
d\tau  \label{1v}
\end{equation}
\[
\leq F_{2}(M,D,C_{2})\sigma \left\| \chi_{11}-\chi_{12}\right\|\leq F_{2}(M,D,C_{2})\ \left\| \vec{\chi_{1}^{*}}-\vec{\chi_{2}^{*}}\right\| \;\sqrt{\sigma} ,
\]

\begin{equation}
\beta\int_{0}^{t} \left|\chi_{21}(\tau)N_{y}(y_{11}(t),t;y_{01}(\tau
),\tau )-\chi_{22}(\tau)N_{y}(y_{12}(t),t;y_{02}(\tau ),\tau
)\right| d\tau  \label{1iiii} 
\end{equation}
\[\leq  F_{3}(M,\beta,C_{1},C_{2})\left\| \vec{\chi_{1}^{*}}-\vec{\chi_{2}^{*}}\right\|\sqrt{\sigma} , \nonumber
\]
\begin{equation}
\beta\int_{0}^{t} |g(\tau)||w_0(\tau)|\left|N_{y}(y_{11}(t),t;y_{01}(\tau
),\tau )-\chi_{22}(\tau)N_{y}(y_{12}(t),t;y_{02}(\tau ),\tau
)\right| d\tau  \label{1iiii} 
\end{equation}
\[\leq  F_{4}(\beta,C_{1},C_{2},g,R)\left\| \vec{\chi_{1}^{*}}-\vec{\chi_{2}^{*}}\right\|\sqrt{\sigma} , \nonumber
\]
\begin{equation}
\int_{0}^{t} \left| \chi_{21}(\tau )G_{\tau}(y_{11}(t),t;y_{01}(\tau
),\tau )-\chi_{22}(\tau )G_{\tau}(y_{12}(t),t;y_{02}(\tau ),\tau
)\right| d\tau  \label{1iiiiii} 
\end{equation}
\[\leq F_{5}(D,M,C_{1},C_{2})\left\| \vec{\chi_{1}^{*}}-\vec{\chi_{2}^{*}}\right\|\sqrt{\sigma},  \nonumber
\]

\begin{equation}
\int_{C_{1}}^{C_{2}}\left| F(\xi )\right| \left|
N(y_{01}(t),t;\xi ,0)-N(y_{02}(t),t;\xi ,0)\right| d\xi
\label{1ii}
\end{equation}
\[
\leq \frac{2\left\|F\right\|_{[C_{1},C_{2}]}}{D\sqrt{\pi}}\left\|
\chi_{2,1}-\chi_{2,2}\right\|_{\sigma} \sqrt{\sigma } \leq F_{6}(u_{0},D,\beta,C_1, C_2) \left\|\vec{\chi_{1}^{*}}-\vec{\chi_{2}^{*}}\right\|\sqrt{\sigma},
\]
\begin{equation}
D\int_{0}^{t}\left| \chi_{11}(\tau )N(y_{01}(t),t;y_{11}(\tau
),\tau )-\chi_{12}(\tau )N(y_{02}(t),t;y_{12}(\tau ),\tau )\right|
d\tau  \label{1iibis}
\end{equation}
\[
\leq  F_{7}(M,C_{1},C_{2})\ \left\| \vec{\chi_{1}^{*}}-\vec{\chi_{2}^{*}}\right\| \;\sqrt{\sigma} ,
\]
\begin{equation}
\beta\int_{0}^{t}\left|\chi_{21}(\tau)
N(y_{01}(t),t;y_{01}(\tau ),\tau
)-\chi_{22}(\tau)N(y_{02}(t),t,y_{0 2}(\tau ),\tau )\right| d\tau  \label{1vi}
\end{equation}
\[
\leq F_{8}(B,M, D,C_{1},C_{2})\left\| \vec{\chi_{1}^{*}}-\vec{\chi_{2}^{*}}\right\| \sqrt{\sigma} ,
\]
\begin{equation}
D\int_{0}^{t}\left| \chi_{21}(\tau )G_{y}(y_{01}(t),t;y_{01}(\tau
),\tau )-\chi_{22}(\tau )G_{y}(y_{02}(t),t;y_{02}(\tau ),\tau )\right|
d\tau  \label{1vbis}
\end{equation}
\[
\leq F_{9}(M,D,C_{1})\ \left\| \vec{\chi_{1}^{*}}-\vec{\chi_{2}^{*}}\right\| \;\sqrt{\sigma},
\]
\begin{equation}
\beta\int_{0}^{t}\left|g(\tau )\right| \left|w_0(\tau)\right|\left|
N(y_{01}(t),t;y_{01}(\tau ),\tau )-N(y_{02}(t),t;y_{02}(\tau ),\tau )\right| d\tau ,\label{iiibisbis}
\end{equation}
\[
\leq F_{10}(B, g, R, M, D,C_{1},C_{2})\left\| \vec{\chi_{1}^{*}}-\vec{\chi_{2}^{*}}\right\| \sqrt{\sigma} ,
\]

where 
\begin{equation}
   F_{0}(u_{0},D,\beta, C_1, C_2)=\tfrac{3D(C_2+C_1)(\beta+1)}{8\sqrt{\pi}\beta^{2}}\left[\left(\tfrac{24}{e(C_2-2C_1)^{2}}\right)^{3/2}+\left(\tfrac{24}{e(C_2+2C_1)^{2}}\right)^{3/2}\right]
\end{equation}
\begin{equation}
F_{1}(u_{0},D,\beta,C_1,C_2)=\tfrac{2}{\sqrt{\pi}}E_1(u_0,C_1,C_2,B,D),
\label{defp1}
\end{equation}
\begin{equation}\label{defp2}
F_{2}(M,D,C_{2})=\tfrac{D}{4\sqrt{\pi}} \left[2(1+\beta)\left(1+\tfrac{M}{\beta^{2}}\right)+3C_2 \left(\tfrac{2}{3e C_2^{2}}\right)^{\frac{3}{2}}\right]
\end{equation}

\[+\tfrac{D^{-1/2} }{\sqrt{\pi}}\left[\tfrac{D(\beta+1)}{\beta^{2}}+2\tfrac{(\beta+1)^{2}}{\beta^{2}}\left(1+\tfrac{DM}{\beta^{2}}\right)^{2}\right]+\left(\tfrac{6}{eC_{2}^{2}}\right)^{3/2}\tfrac{18C_{2}^{2}+1}{4\sqrt{\pi}}\tfrac{2D(\beta+1)}{\beta^{2}}
\]

\begin{equation}
F_{3}(M,\beta,C_{1},C_{2})=M\beta(F_{31}+F_{32}),
\end{equation} with
\begin{equation}
F_{31}(C_{1},C_{2})=\tfrac{1}{\sqrt{\pi }e^{3/2}}\left[ \tfrac{\sqrt{6}\left(
3C_{2}-C_{1}\right) ^{2}}{16(C_{2}-3C_{1})^{3}}+\tfrac{27\sqrt{3}}{4}+\tfrac{12%
\sqrt{6}}{(C_{2}-3C_{1})^{3}}+\tfrac{6\sqrt{3}}{(C_{2}+C_{1})^{3}}\right],
\label{defp31}
\end{equation}
\begin{equation}
F_{32}(C_{1},C_{2})=\tfrac{12\sqrt{6}}{\sqrt{\pi }e^{3/2}}\left[ \tfrac{1}{%
(C_{2}-3C_{1})^{3}}+\tfrac{9}{8}+\tfrac{\left( 3C_{2}-C_{1}\right) ^{2}}{%
8(C_{2}-3C_{1})^{3}}+\tfrac{1}{(C_{2}+C_{1})^{2}}\right],  \label{defp32}
\end{equation}

\begin{equation}F_{4}(\beta,C_{1},C_{2},g,R)=\beta R \|g\|(F_{31}+F_{32}),\end{equation}
\begin{equation}
F_{5}(M,C_{1},C_{2})=\left[M(F_{31}+F_{32})+F_{51}\right],
\end{equation} where
\begin{equation}
F_{51}(C_{1},C_{2})=\tfrac{\sqrt{6}}{\sqrt{\pi e}}\left[ \tfrac{1}{%
(C_{2}-3C_{1})^{2}}+\tfrac{1}{(C_{2}+C_{1})^{2}}\right]  ,\label{defp3}
\end{equation}
\begin{equation}F_{6}(u_{0},D,\beta,C_1, C_2)=\tfrac{2}{D\sqrt{\pi}} \exp\left(\tfrac{\left\|(u_{0}\right\|+\beta)(C_2-C_1)}{D}\right) (\|u_{0}\|+\beta),
\end{equation}
\begin{equation}
F_{7}(\beta, M,D,C_{1},C_{2})=\beta M\tfrac{6^{3/2}D}{\sqrt{\pi }e^{3/2}}\left[ \frac{%
3C_{2}-C_{1}}{(C_{2}-3C_{1})^{3}}+\frac{3}{(C_{2}+C_{1})^{2}}\right],
\label{defp5}
\end{equation}

\begin{equation}
F_{8}(M,D,C_{1})=\tfrac{\sqrt{D}}{4\sqrt{\pi }}\left[ 6M+\tfrac{3}{C_{1}^{2}}\left( \tfrac{2%
}{3e}\right) ^{3/2}+\tfrac{6M}{C_{1}^{2}}\left( \tfrac{6}{e}\right) ^{3/2}\right].
\label{defp2}
\end{equation}
\begin{equation}\label{defp2}
F_{9}(M,\beta,H,D,C_{1})=\tfrac{D}{4\sqrt{\pi}} \left[2(1+\beta)\left(1+\tfrac{M}{\beta^{2}}\right)+3C_1 \left(\tfrac{2}{3eC_1^{2}}\right)^{3/2}\right]
\end{equation}

\[+\tfrac{D^{-1/2} }{\sqrt{\pi}}\left[\tfrac{1}{H}+2\tfrac{1}{H}\left(\beta+D\|g\|+\tfrac{M}{H}\right)^{2}\right]+\left(\tfrac{6}{eC_{1}^{2}}\right)^{3/2}\tfrac{18C_{1}^{2}+1}{4\sqrt{\pi}}\tfrac{1}{H}
\]
\begin{equation}
F_{10}(B, g, R, M, D,C_{1})= \beta R \|g\|\left(\tfrac{1}{\sqrt{\pi D}}\left(\beta+D\|g\|+\tfrac{M}{H}\right)^{2}+\left(\tfrac{6}{e}\right)^{3/2}\tfrac{1}{C_1^{2}\sqrt{\pi}}\right)
\label{defp2bis}
\end{equation}
\end{lemma}

\begin{proof}  The inequalities are obtained following \cite{BrNa2012},\cite{BrTa2006}  and \cite{Sh}.

We will show the proof of how the inequality \eqref{1v} is obtained.

Firstly we have
\begin{equation*}
D\int_{0}^{t}\left| \chi_{11}(\tau )N_{y}(y_{11}(t),t;y_{11}(\tau
),\tau )-\chi_{12}(\tau )N_{y}(y_{12}(t),t;y_{12}(\tau ),\tau )\right|
d\tau  
\end{equation*}
\begin{equation*}
\leq D\int_{0}^{t}\left| \chi_{11}(\tau )-\chi_{12}(\tau )\right| |N_{y}(y_{11}(t),t;y_{11}(\tau
),\tau )|d\tau
\end{equation*}
\begin{equation}\label{1vprima} 
+ D\int_{0}^{t}|\chi_{12}(\tau )| |N_{y}(y_{11}(t),t;y_{11}(\tau
),\tau )-N_{y}(y_{12}(t),t;y_{12}(\tau ),\tau )|d\tau
\end{equation}
We write 
\[
\left| N_{y}(y_{11}(t),t;y_{11}(\tau ),\tau
)-N_{y}(y_{12}(t),t;y_{12}(\tau ),\tau )\right| 
\]
\[
\leq \left|
K_{y}(y_{11}(t),t;y_{11}(\tau ),\tau )-K_{y}(y_{12}(t),t;y_{12}(\tau ),\tau )\right|\] 
\[+\left|
K_{y}(-y_{11}(t),t;y_{11}(\tau ),\tau )-K_{y}(-y_{12}(t),t;y_{12}(\tau ),\tau )\right|. 
\]
We have
\[
\left|
K_{y}(y_{11}(t),t;y_{11}(\tau ),\tau )-K_{y}(y_{12}(t),t;y_{12}(\tau ),\tau )\right|\] 
\[
\leq(2D(t-\tau))^{-1} \left|K(y_{11}(t),t;y_{11}(\tau ),\tau )\left[\left(y_{11}(t)-y_{11}(\tau )\right)-\left(y_{12}(t)-y_{12}(\tau )\right)\right]\right. 
\]
\[
+\left.
\left[K(y_{11}(t),t;y_{11}(\tau ),\tau )-K(y_{12}(t),t;y_{12}(\tau ),\tau )\right]\left(y_{12}(t)-y_{12}(\tau )\right)\right|\] 
\[
\leq (2D(t-\tau))^{-1} K(y_{11}(t),t;y_{11}(\tau ),\tau )
\left|\left[\left(y_{11}(t)-y_{11}(\tau )\right)-\left(y_{12}(t)-y_{12}(\tau )\right)\right]\right.
 \]
\[
\left.+\left[1-exp(m(t,\tau)\right]\left(y_{12}(t)-y_{12}(\tau )\right)
\right|, 
\]
where
\[
m(t,\tau)=\frac{\left(y_{11}(t)-y_{11}(\tau )\right)^{2}-\left(y_{12}(t)-y_{12}(\tau )\right)^{2}}{4D(t-\tau)}
\]
\[=\frac{\left[\left(y_{11}(t)-y_{01}(\tau )\right)-\left(y_{12}(t)-y_{12}(\tau )\right)\right]\left[\left(y_{11}(t)-y_{11}(\tau )\right)+\left(y_{12}(t)-y_{12}(\tau )\right)\right]}{4D(t-\tau)}.
\]
and from \eqref{ese}
\[
\left|\left(y_{11}(t)-y_{11}(\tau )\right)-\left(y_{12}(t)-y_{12}(\tau )\right)\right| \leq \tfrac{D(\beta+1)}{\beta^{2}}\int_{\tau}^{t}\left|\chi_{11}(\eta)-\chi_{12}(\eta)\right| d\eta 
\]
\[
\leq \tfrac{D(\beta+1)}{\beta^{2}}\left\| \vec{\chi_{1}^{*}}-\vec{\chi_{2}^{*}}\right\|(t-\tau),
\]
and 
\[
\left|\left(y_{11}(t)-y_{11}(\tau )\right)+\left(y_{12}(t)-y_{12}(\tau )\right)\right| \leq 2(1+\beta)\left(1 + \tfrac{DM}{\beta^{2}}\right)(t-\tau)\leq 4(1+\beta)\left(1 + \tfrac{DM}{\beta^{2}}\right)\sigma,
\]
then we have 
\[
\left| m(t,\tau)\right|\leq\tfrac{\beta+1}{\beta^{2}}\left(1+\tfrac{DM}{\beta^{2}}\right)\left\| \vec{\chi_{1}^{*}}-\vec{\chi_{2}^{*}}\right\|\sigma.
\]
In addition, by using $\left\| \vec{\chi_{1}^{*}}-\vec{\chi_{2}^{*}}\right\|\leq 2M$ we have
\[
\left| m(t,\tau)\right|\leq 2M\tfrac{\beta+1}{\beta^{2}}\left(1+\tfrac{DM}{\beta^{2}}\right)\sigma,
\]
assuming 
\[2M\tfrac{\beta+1}{\beta^{2}}\left(1+\tfrac{DM}{\beta^{2}}\right)\sigma\leq 1,
\] and the fact that
$|1-exp(x)|\leq (1-x)$ for $x<1$
we obtain 
\[\left|1-exp(m(t,\tau)\right|\leq 2\left| m(t,\tau)\right|\leq 2\tfrac{\beta+1}{\beta^{2}}\left(1+\tfrac{DM}{\beta^{2}}\right)\left\| \vec{\chi_{1}^{*}}-\vec{\chi_{2}^{*}}\right\|\sigma.
\]
Therefore
\[
\left|
K_{y}(y_{11}(t),t;y_{11}(\tau ),\tau )-K_{y}(y_{12}(t),t;y_{12}(\tau ),\tau )\right|\leq 
\] 
\[
\leq (2D)^{-1} \tfrac{1}{\sqrt{D\pi(t-\tau)}}\left[\tfrac{D(\beta+1)}{\beta^{2}}+2\tfrac{(\beta+1)^{2}}{\beta^{2}}\left(1+\tfrac{DM}{\beta^{2}}\right)^{2}\sigma\right]\left\| \vec{\chi_{1}^{*}}-\vec{\chi_{2}^{*}}\right\|
\]

Moreover, by Cauchy's mean value theorem we have  
\[\left|
K_{y}(-y_{11}(t),t;y_{01}(\tau ),\tau )-K_{y}(-y_{12}(t),t;y_{12}(\tau ),\tau )\right| 
\]
\[\leq \left|
K(n(t,\tau),t;0,\tau)\left( \frac{n^{2}(t,\tau)}{4D^{2}(t-\tau)^{2}}-\frac{1}{2D(t-\tau)}\right)\right| \left|y_{11}(t)+y_{11}(\tau)-y_{12}(t)- y_{12}(\tau)\right|
\]
where $n=n\left(
t,\tau\right)$ is between $y_{11}(t)+y_{11}(\tau)\;$ and $y_{12}(t)+ y_{12}(\tau)$.

Since
\[ \left|y_{11}(t)+y_{11}(\tau)-y_{12}(t)- y_{12}(\tau)\right|\leq \tfrac{2D(\beta+1)}{\beta^{2}}\sigma \left\| \vec{\chi_{1}^{*}}-\vec{\chi_{2}^{*}}\right\|,
\]
and $
C_{2}\leq n\left(
t,\tau\right)\leq 6 C_{2}$, by using $(\ref{exp})$ we have
\[\left|
K(n(t,\tau),t;0,\tau)\left( \tfrac{n^{2}(t,\tau)}{4D^{2}(t-\tau)^{2}}-\frac{1}{2D(t-\tau)}\right)\right|\leq\left(\tfrac{6}{eC_{2}^{2}}\right)^{3/2}\tfrac{18C_{2}^{2}+1}{4\sqrt{\pi}},
\]
then
\[\left|
K_{y}(-y_{11}(t),t;y_{11}(\tau ),\tau )-K_{y}(-y_{12}(t),t;y_{12}(\tau ),\tau )\right| \leq \left(\tfrac{6}{eC_{2}^{2}}\right)^{3/2}\tfrac{18C_{2}^{2}+1}{4\sqrt{\pi}}\tfrac{2D(\beta+1)}{\beta^{2}}\sigma \left\| \vec{\chi_{1}^{*}}-\vec{\chi_{2}^{*}}\right\|.
\]
In this way we can affirm that
\[
\left| N_{y}(y_{11}(t),t;y_{11}(\tau ),\tau
)-N_{y}(y_{12}(t),t;y_{12}(\tau ),\tau )\right| 
\]
\[
\leq \left\lbrace\tfrac{D^{-3/2} }{2\sqrt{\pi(t-\tau)}}\left[\tfrac{D(\beta+1)}{\beta^{2}}+2\tfrac{(\beta+1)^{2}}{\beta^{2}}\left(1+\tfrac{DM}{\beta^{2}}\right)^{2}\sigma\right]+\left(\tfrac{6}{eC_{2}^{2}}\right)^{3/2}\tfrac{18C_{2}^{2}+1}{4\sqrt{\pi}}\tfrac{2D(\beta+1)}{\beta^{2}}\sigma\right\rbrace \left\| \vec{\chi_{1}^{*}}-\vec{\chi_{2}^{*}}\right\|.
\]
Finally, for \eqref{1vprima} we conclude that \begin{equation*}
D\int_{0}^{t}\left| \chi_{11}(\tau )N_{y}(y_{11}(t),t;y_{11}(\tau
),\tau )-\chi_{12}(\tau )N_{y}(y_{12}(t),t;y_{12}(\tau ),\tau )\right|
d\tau  
\end{equation*}
\[
\leq \lbrace\tfrac{D}{4\sqrt{\pi}} \left[2(1+\beta)\left(1+\tfrac{M}{\beta^{2}}\right)+3C_2 \left(\tfrac{2}{3eC_2^{2}}\right)^{3/2}\right]\]

\[+\tfrac{D^{-1/2} }{\sqrt{\pi}}\left[\tfrac{D(\beta+1)}{\beta^{2}}+2\tfrac{(\beta+1)^{2}}{\beta^{2}}\left(1+\tfrac{DM}{\beta^{2}}\right)^{2}\right]+\left(\tfrac{6}{eC_{2}^{2}}\right)^{3/2}\tfrac{18C_{2}^{2}+1}{4\sqrt{\pi}}\tfrac{2D(\beta+1)}{\beta^{2}}\rbrace\sqrt{\sigma} \| \vec{\chi_{1}^{*}}-\vec{\chi_{2}^{*}}\|
\]

\[= F_{2}(D,\beta,M,C_{2})\left\| \vec{\chi_{1}^{*}}-\vec{\chi_{2}^{*}}\right\|\sqrt{\sigma},
 \]

\end{proof}

\begin{theorem}\label{teocont} Let hypothesis $(\ref{hip})$ be. Fixed $D<2$, $0<3C_{1}<C_2$ and $w_0\in \Lambda$. If $\sigma $ satisfies the following inequalities 

\begin{equation}
\sigma \leq \tfrac{1}{M},\,\quad 2(1+\beta)\left(1+\frac{DM}{\beta^{2}}\right)\sigma \leq C_{2},\quad\quad \; \label{nose}
\end{equation}
\begin{equation} 
2\left(\beta+D\|g\|+\tfrac{M}{H}\right)\sigma\leq C_{1},\quad\quad ,
\end{equation}
\begin{equation}
    \tfrac{M}{HD}\left(\beta + D\|g\|+\tfrac{M}{H}\right)\sigma \leq 1
\end{equation}
\begin{equation}
    \tfrac{2M(\beta+1)}{\beta^{2}}\left(1 + \tfrac{DM}{\beta^{2}}\right)\sigma \leq 1
\end{equation}
\begin{equation}
H_{1}\left( C_{1},u_0, C_2,M,D,H,R, \beta,\sigma\right) \leq 1,
\label{ache1}
\end{equation}
\begin{equation}
H_{2}\left(C_{1},C_2, u_0,g,M,D,\beta,\sigma  \right)\leq 1,  \label{ache2}
\end{equation}

where $M$ is given by 
\begin{equation}
M=1+\frac{2}{2-D}(E_{1}+E_6), \label{erre}
\end{equation}
and 
\begin{equation}
H_{1}\left( C_{1},u_0, C_2,M,D,H,R, \beta,\sigma\right) =\tfrac{2}{2-D}\left[E_{0}+E_{2}+E_{3}+E_{4}+E_5+E_7\right.\label{defache1}\end{equation}
\[
\left.+E_8+E_9+E_{10}\right]\sqrt{\sigma},
\]
\begin{equation}
H_{2}\left(C_{1},C_2, u_0,g,M,D,\beta,\sigma  \right) = \frac{2}{2-D}
\sum_{i=0}^{10}F_{i}
\label{defache2}
\end{equation}
then the map $\Psi:C_{M,\sigma }\longrightarrow C_{M,\sigma }$ is well defined and it is a contraction map. Therefore there exists a unique solution $\chi_{1}^{*}$, $\chi_{2}^{*}$ on $C_{M,\sigma }$ to the system of
integral equations (\ref{ecintegralf}) and (\ref{ecintegralf1}).
\end{theorem}

\begin{proof}
Firstly, we demonstrate that $\Psi
$ maps $C_{M,\sigma }\;$into itself, that is 
\[
\left\| \Psi\left( \stackrel{\longrightarrow }{\chi^{*}}\right) \right\| _{\sigma
}=\max \limits_{t\in \left[ 0,\sigma \right] }\left|
\Psi_{1}(\chi_{1}(t),\chi_{2}(t))\right| +\max \limits_{t\in \left[ 0,\sigma \right] }
\left| \Psi_{2}(\chi_{1}(t),\chi_{2}(t))\right| \leq M 
\]
From Lemma \ref{cotasint} we have 
\[
\left| \Psi_{1}(\chi_{1}(t),\chi_{2}(t))\right| \leq \frac{2}{2-D} \left\lbrace E_{1}+\left[E_{0}+E_{2}+E_{3}+E_{4}+E_5\right]\sqrt{\sigma}\right\rbrace,
\]
\[
\left| \Psi_{2}(\chi_{1}(t),\chi_{2}(t))\right| \leq \frac{2 }{2-D}\left\lbrace E_{6}+\left[ E_7+E_8+E_9+E_{10}\right]\sqrt{\sigma}\right\rbrace,
\]
therefore 
\[
\left\| \Psi\left( \stackrel{\longrightarrow }{\chi^{*}}\right) \right\| _{\sigma
}\leq   \frac{2}{2-D}(E_1+E_6)+H_{1}
\]
where $H_{1}$ is given by $\left( \ref{defache1}\right).$ 

Selecting $M$ by $%
\left( \ref{erre}\right) \;$and $\sigma \;$such that $\left( \ref{ache1}%
\right) \;$holds, we obtain $\left\| \Psi\left( \stackrel{\longrightarrow }{\chi^{*}%
}\right) \right\| _{\sigma }\leq M.$ 

Taking into account Lemma \ref{cotasintyi} we have 
\[
\left\| \Psi\left( \stackrel{\longrightarrow }{\chi_{1}^{*}}\right) -\Psi\left( 
\stackrel{\longrightarrow }{\chi_{2}^{*}}\right) \right\| _{\sigma }\leq
H_{2}\left( C_{1},C_{2},u_0,M,D,R, H,\beta,\sigma \right) \left\| 
\stackrel{\longrightarrow }{\chi_{1}^{*}}-\stackrel{\longrightarrow }{\chi_{2}^{*}}%
\right\| _{\sigma } 
\]
where $\stackrel{\longrightarrow }{\chi_{1}^{*}}=\binom{\chi_{11}}{\chi_{12}}\;,\;%
\stackrel{\longrightarrow }{\chi_{2}^{*}}=\binom{\chi_{21}}{\chi_{22}}$ $\in
C_{M,\sigma }$. 
Then, assuming (\ref{nose})-(\ref{erre}) we have that $\Psi$ is a contraction and we can assure that there exists a unique fixed point $\chi^{*}=\binom{\chi_{1}}{\chi_{2}}$ such that $\Psi(\chi^{*})=\chi^{*}$ this is 
$$\Psi_{1}(\chi_{1}(t),\chi_{2}(t))=\chi_{1}(t),\quad\quad \Psi_{2}(\chi_{1}(t),\chi_{2}(t))=\chi_{2}(t). $$ Therefore the system (\ref{ecintegralf}) and (\ref{ecintegralf1}) has a unique solution.
\end{proof}

\begin{corollary}

    For each function $w_0\in\Lambda$, there exists a unique integral representation for $w$, $y_{0}$ and $y_{1}$ given by $(\ref{z})$, $(\ref{ycero})$ and $(\ref{ese})$ respectively, where $\chi_{1}$ and $\chi_{2}$ are the unique solutions of (\ref{ecintegralf}) and (\ref{ecintegralf1}). This representation depends on such function $w_0$.
    
\end{corollary}
Next section we will analyze the existence of solution of the integral equation \eqref{w0} given by
\begin{equation}\label{w0sol}
w_{0}(t)=1-\int^{t}_{0} \chi^{w_0}_{1}(\tau)d\tau+\int^{y^{w_0}_{1}(t)}_{y^{w_0}_{0}(t)}w^{w_0}(\xi,t)d\xi.
\end{equation}
We have used the notation $\chi^{w_0}_{1}$, $y^{w_0}_{1}$, $y^{w_0}_{1}$ and $w^{w_0}$ to indicate their dependency of $w_0$.

Assuming hypothesis of Theorem \ref{teocont} we will prove that for suitable $H$, $R$ and $\sigma$ there exists a unique $w_0\in\Lambda$ such that \eqref{w0sol} is satisfied.
 
We define the operator $\varphi$ as follow
\begin{equation}\label{defPsi}
  \varphi(w_0)(t)=1-\int^{t}_{0} \chi^{w_0}_{1}(\tau)d\tau+\int^{y^{w_0}_{1}(t)}_{y^{w_0}_{0}(t)}w^{w_0}(\xi,t)d\xi
\end{equation}
where $\chi^{w_0}_{1}$, $w^{w_0}$, $y^{w_0}_{0}$ and $y^{w_0}_{1}$ are the functions obtained from Theorem \ref{teocont} for each $w_0\in \Lambda$.
Proving that  equation \eqref{w0sol} has a solution is equivalent to proving that the operator $\varphi$ has a fixed point.
For this, we will use the Schauder's fixed point theorem which states: If $\varphi$ is a continuous operator of a convex and compact set $\Lambda$ to itself then will exist $w_0$ such that  $\varphi(w_{0})=w_{0}.$

\begin{lemma}\label{propZh} If $w_0\in \Lambda$ then function $\varphi(w_0)\in C[0,\sigma]$ and the following inequalities holds
\begin{equation}
\varphi(w_0)(t)\geq 1,\label{H}
\end{equation}
\begin{equation}
\varphi(w_0)(t)\leq 
1+M\sigma+M(\tfrac{3C_2}{2}-\tfrac{C_1}{2})\leq 2+M(\tfrac{3C_2}{2}-\tfrac{C_1}{2})
\end{equation}\label{R}
\end{lemma}
\begin{proof} The continuity of the function $\varphi(w_0)$ immediately leaves its definition.  
Taking into account that by the principle of the maximum the solution $w^{w_0}$ to the problem \eqref{calu}-\eqref{free} satisfies $w^{w_0}(t)\geq 0$ and $\chi^{w_0}_{1}(t)<0$ we can assure that $\varphi(w_0)(t)\geq 1$ for all $t\in[0,\sigma]$. 
Moreover, we have
\[\varphi(w_0)(t)\leq 1+Mt+max\lbrace{w^{w_0}(y^{w_0}_0(t),t):t\in[0,\sigma]\rbrace}|y_1^{w_0}(t)-y_0^{w_0}(t)|
\]
\[
=1+Mt+max\lbrace{\chi_2(t):t\in[0,\sigma]\rbrace}|y_1^{w_0}(t)-y_0^{w_0}(t)|
\]
\[
\leq1+M\sigma+M(\tfrac{3C_2}{2}-\tfrac{C_1}{2})
\]
and the lemma holds.
\end{proof}

\begin{lemma}\label{ZhenPi}
Assuming 
$H=1$, $R=2+M(\tfrac{3C_2}{2}-\tfrac{C_1}{2})$
and hypothesis of Theorem \ref{teocont} we have $\varphi(w_0)\in\Lambda.$
\end{lemma}
\begin{proof} 
It follows immediately from the definition of $\Lambda$ and the previous lemma.

\end{proof}
\begin{theorem} Under the hypothesis of Lemma $\ref{ZhenPi}$ there exists $w_0^{*}\in \Lambda$ such that $\varphi(w_0^{*})=w_0^{*}$. 
\end{theorem}
\begin{proof} 
The proof continues using Schauder's theorem, the previous lemmas and the fact that $\Lambda$ is a convex and compact subset of $C[0,\sigma]$.
\end{proof}
Finally we are in a position to announce the main result.

\begin{theorem} \label{teoremafinal}Fixed $0<D < 2$, $0 < 3C_1 < C_2$, $M=1+\tfrac{2}{2-D}(E_{1}+E_6)$, $H=1$ and $R=2+M(\tfrac{3C_2}{2}-\tfrac{C_1}{2}) .$
If $\sigma $ satisfies the following inequalities 

\begin{equation}
\sigma \leq \tfrac{1}{M},\,\quad 2(1+\beta)\left(1+\frac{DM}{\beta^{2}}\right)\sigma \leq C_{2},\quad\quad \; \label{nose}
\end{equation}
\begin{equation} 
2\left(\beta+D\|g\|+\tfrac{M}{H}\right)\sigma\leq C_{1},\quad\quad ,
\end{equation}
\begin{equation}
    \tfrac{M}{HD}\left(\beta + D\|g\|+\tfrac{M}{H}\right)\sigma \leq 1
\end{equation}
\begin{equation}
    \tfrac{2M(\beta+1)}{\beta^{2}}\left(1 + \tfrac{DM}{\beta^{2}}\right)\sigma \leq 1
\end{equation}
\begin{equation}
\tfrac{2}{2-D}\left[E_{0}+E_{2}+E_{3}+E_{4}+E_5+E_7+E_8+E_9+E_{10}\right]\sqrt{\sigma} \leq 1,
\label{ache1}
\end{equation}
\begin{equation}
\frac{2}{2-D}
\sum_{i=0}^{10}F_{i}\sqrt{\sigma}\leq 1,  \label{ache2}
\end{equation}
then the map $\Psi:C_{M,\sigma }\longrightarrow C_{M,\sigma }$ is well defined and it is a contraction map. Therefore there exists a unique solution $\chi_{1}^{*}$, $\chi_{2}^{*}$ on $C_{M,\sigma }$ to the system of
integral equations (\ref{ecintegralf}) and (\ref{ecintegralf1}),
corresponding to $w_0^{*}\in \Lambda(1,R)$ which is the solution of $\eqref{w0sol}.$

\end{theorem}

\begin{corollary}\label{corolultimo}
    From the results given in Theorem \ref{teocont} and  Theorem \ref{teoremafinal} we obtain that there exist  ($w^{*}(y,t)$, $y^{*}_{0}(t)$, $y^{*}_{1}(t)$) the solution to the problem $\eqref{calu}-\eqref{free}$ which is given by $(\ref{z})$,  $(\ref{ycero})$ and $(\ref{ese})$ respectively, with $\chi_{1}=\chi_{1}^{*}$, $\chi_{2}=\chi_{2}^{*}$ the unique solutions to  (\ref{ecintegralf})\ and (\ref{ecintegralf1}) corresponding to the solution $w_0^{*}$ to the equation $(\ref{w0sol}).$
\end{corollary}
    
In order to extend the solution given by the Theorem \ref{teoremafinal} and Corollary \ref{corolultimo} for all time we take $\sigma*$ the maximum of $\sigma$ such that the hypothesis of this Theorem \ref{teoremafinal} hold. 
\begin{remark}\label{obsul} (extended solution)
If we consider the free boundary problem given by equations \eqref{calu}-\eqref{free} for $t\geq \sigma^{*}$, redefining time as $\eta=t-\sigma^{*}\geq 0$ and assuming the condition
\begin{equation}
w(y,0)=w^{*}(y,\sigma^{*}),\qquad y^{*}_0(\sigma^{*})<y< y^{*}_1(\sigma^{*})
\end{equation}
as initial condition instead of the one given in \eqref{tt21} with $w^{*}$ given by Corollary \ref{corolultimo}, we can to solve the free boundary problem,  following the same methodology, and so to obtain that there exists solution for $\sigma^{*}<t<\sigma_1$, to so on.
\end{remark}

\begin{theorem} \label{teolocal}
  
Assuming hypothesis of Theorem $\ref{teoremafinal}$, we obtain there exist solution to the free boundary problem $(\ref{calor1})-\ref{tempborde})$ which admits the following explicit parametric representation 
\begin{equation}
u^{*}(x,t)=\frac{w^{*}(y,t)}{1-\int_{0}^{t}\chi^{*}_{1}(\tau) d\tau+\frac{1}{D}\int_{y}^{y^{*}_{1}(t)} w^{*}(\xi,t)d\xi}+\beta \quad
\end{equation}
\begin{equation}
x=\int_{y^{*}_{0}(t)+2\beta t}^{y+2\beta t} \left[\frac{w^{*}(\mu,t)}{1-\int_{0}^{t}\chi^{*}_{1}(\tau) d\tau+\frac{1}{D}\int_{\mu}^{y^{*}_{1}(t)} w^{*}(\xi,t)d\xi}+\beta\right] d\mu
\end{equation}
with
\[
y^{*}_{0}(t)\leq y\leq y^{*}_{1}(t),\quad\quad t>0
\]
and
\begin{equation}
s(t)=\int_{y^{*}_{0}(t)+2\beta t}^{y^{*}_{1}(t)+2\beta t} \left[\frac{w^{*}(\mu,t)}{1-\int_{0}^{t}\chi^{*}_{1}(\tau) d\tau+\frac{1}{D}\int_{\mu}^{y^{*}_{1}(t)} w^{*}(\xi,t)d\xi}+\beta\right] d\mu
\end{equation}
where ($w^{*}(y,t)$, $y^{*}_{0}(t)$, $y^{*}_{1}(t)$) is the solution to the problem $\eqref{calu}-\eqref{free}$.

\end{theorem}
\begin{proof}
    We invert the transformations given by $(\ref{firsttrans})$, $(\ref{2})$ and $(\ref{tres})$ and use Corollary \ref{corolultimo} and Remark \ref{obsul}.
\end{proof}

\end{section}

\section*{Statements and Declarations}
The author declare that she has no known competing financial interests
or personal relationships that could have appeared to influence the work
reported in this paper.

\section*{ACKNOWLEDGEMENT}
The present work has been partially sponsored by the Project PIP No 11220220100532CO CONICET-UA and the Project  O06-24CI1901 Universidad Austral, Rosario, Argentina



\section*{REFERENCES}
\begin{thebibliography}{00}

\bibitem[AlSo(1993)]{AlSo} V. Alexiades, A. D. Solomon, Mathematical Modelling of Melting and Freezing Processes, Hemisphere-Taylor Francis, Washington (1993).

\bibitem[BrNa(2012)]{BrNa2012} A. C. Briozzo and M. F. Natale , \textit{On a non-linear moving boundary problem for a diffusion–convection equation}, International Journal of Non-Linear Mechanics, \textbf{47} (2012), 712-718.


\bibitem[BrTa(2006)]{BrTa2006} A. C. Briozzo and D. A. Tarzia, \textit{Existence and uniqueness for one-phase Stefan problems of non-classical heat equations with temperature boundary condition at a fixed face.} Electronic Journal of Differential Equations. 
\textbf{21} (2006), 1--16.

\bibitem[BrTa(2019)]{BrTa2019} A. C. Briozzo and D. A. Tarzia, \textit{On the paper D. Burini-S. De Lillo-G. Fioriti, Acta Mech., 229 No. 10 (2018), pp.4215-4228}, Acta Mechanica, Volume 231, pages 391–393, (2020).

\bibitem[BrTa(2020)]{BrTa2020} A. C. Briozzo and D. A. Tarzia, \textit{A free boundary problem for a diffusion-convection
equation}, International Journal of Non-Linear Mechanics, Volume 120, April 2020, 103394.

\bibitem[BuDeL(2012)]{BuDeLi2012} D. Burini, S. De Lillo , \textit{The Dirichlet-to-Neumann map for a nonlinear diffusion–convection equation} J. Phys. A: Math. Theor.\textbf{45} (2012) 405201.

\bibitem[BuDeL(2018)]{BuDeLiFi2018} D. Burini, S. De Lillo and G. Fioriti, \textit{Nonlinear diffusion in arterial tisues: a free boundary problem.} Acta Mech. \textbf{229} No. 10 (2018), 4215-4228.


\bibitem[Ca(1984)]{Ca}  J. R. Cannon, The one-dimensional heat equation, Menlo Park, CA: Addison-Wesley, 1984.

\bibitem[FoYo(1982)]{FoYo} A. S. Fokas and Y. C. Yorstos, \textit{On the exactly solvable equation} $S_{t} = \left[(\beta S+\gamma)^{-2}-2 S_{x}\right]_{x} +\alpha(\beta S+\gamma)^{-2} S_{x}$ \textit{occurring intwo-phase flow in porous media}, SIAM J. Appl. Math. \textbf{42}, (1982), 318-331. 

\bibitem[Fr(1959)]{Fr1959} A. Friedman, {\em Free boundary problems for parabolic equations I. Melting of solids.} J. Math. Mech.  \textbf{8} (1959), 499--517.

\bibitem[Mi(1971)]{Mi}  R. K. Miller, \textit{Volterra integral equations}, W.A. Benjamin, Menlo Park (1971).

\bibitem[PoMo(2007)]{PoMo} G. Pontrelli and F. de Monte, F.\textit{ Mass diffusion through two-layer porous media: an application to the drug-eluting stent}. Int. J.
Heat Mass Transf. 50, 3658–3669 (2007)

\bibitem[RoBr(2020)]{RoBr2020}C. Rogers and P. Broadbridge \textit{ On transport through heterogeneous media: application of conjugated reciprocal transformations} Z. Angew. Math. Phys. (2020) 71:86 

\bibitem[Ro(1982)]{Ro}G. Rosen, \textit{Method for the exact solution of a nonlinear diffusion-convection equation} Phys. Rev. Lett. \textbf{49}, (1982) 1844.

\bibitem[Ru(1967)]{Ru} L. I. Rubinstein, {\em The Stefan problem.}  Trans. Math. Monographs, Vol. 27, Amer. Math. Soc., Providence, 1971

\bibitem[Sh(1967)]{Sh}  B. Sherman, \emph{A free boundary problem for the heat
equation with prescribed flux at both fixed face and melting interface, } Quart. Appl. Math., \textbf{25} (1967), 53-63.


\bibitem[Ta(2000)]{Ta} D. A. Tarzia, \textit{A bibliography on moving-free boundary problems for the heat diffusion equation. The Stefan and related problems}, MAT - Serie A, Rosario, \textbf{2} (2000), 1-297.

\bibitem[WWE(2001)]{WWE} Wang, C.W., Wu, D., Edelman, E.R.\textit{Physiological transport forces govern drug distribution for stent-based delivery} Circulation 104, 600–605 (2001)
\end{thebibliography}
\end{document}